\documentclass[reqno,10pt]{amsart}

\usepackage{amsmath,amsthm,amssymb,mathtools,amscd}
\usepackage{latexsym,graphicx,color,indentfirst}
\usepackage[linesnumbered,boxed,norelsize]{algorithm2e}
\usepackage{fancyhdr}
\usepackage{mathrsfs}
\usepackage{enumitem}
\usepackage{appendix}

\makeatletter \@addtoreset{equation}{section} \makeatother
\newtheorem{theorem}{Theorem}[section]
\newtheorem{lemma}[theorem]{Lemma}
\newtheorem{assumption}{Assumption}
\newtheorem{Rule}{Rule}

\newtheorem{proposition}[theorem]{Proposition}

\newtheorem{remark}[theorem]{Remark}
\newtheorem{example}[theorem]{Example}

\def\yd{y^\delta}

\def\p{\partial}
\def\d{\delta}
\def\ep{\varepsilon}
\def\l{\langle}
\def\r{\rangle}
\def\E{{\mathcal E}}
\def\R{\mathcal R}

\def\a{\alpha}
\def\la{\lambda}

\begin{document}

\title[Dual gradient flow for ill-posed problems]{Dual gradient flow for solving linear ill-posed problems in Banach spaces}

\author{Qinian Jin}
\address{Mathematical Sciences Institute, Australian National
University, Canberra, ACT 2601, Australia}
\email{qinian.jin@anu.edu.au} \curraddr{}

\author{Wei Wang}
\address{College of Data Science, Jiaxing University, Jiaxing, Zhejiang Province, 314001,  China}
\email{weiwang@zjxu.edu.cn}



%
%


\begin{abstract}
We consider determining the $\R$-minimizing solution of ill-posed problem $A x = y$ for a bounded 
linear operator $A: X \to Y$ from a Banach space $X$ to a Hilbert space $Y$, where $\R: X \to (-\infty, \infty]$ 
is a strongly convex function. A dual gradient flow is proposed to approximate the sought solution by using noisy data.
Due to the ill-posedness of the underlying problem, the flow demonstrates the semi-convergence phenomenon and 
a stopping time should be chosen carefully to find reasonable approximate solutions. We consider the choice of a proper 
stopping time by various rules such as the {\it a priori} rules, the discrepancy principle, and the heuristic discrepancy 
principle and establish the respective convergence results. Furthermore, convergence rates are derived under the variational 
source conditions on the sought solution. Numerical results are reported to test the performance of the dual gradient flow. 
\end{abstract}

\maketitle

\section{\bf Introduction}

Consider the linear ill-posed problem
\begin{equation}\label{1.1}
Ax = y,
\end{equation}
where $A: X \to Y$ is a bounded linear operator from a Banach space $X$ to a Hilbert space $Y$. 
Throughout the paper we always assume (\ref{1.1}) has a solution. The equation (\ref{1.1}) may 
have many solutions. By taking into account {\it a priori} information about the sought solution, 
we may use a proper, lower semi-continuous, convex function $\R: X \to (-\infty, \infty]$ to select 
a solution $x^\dag$ of (\ref{1.1}) such that
\begin{align}\label{R_min}
\R(x^\dag) = \min\{\R(x): x \in X \mbox{ and } A x = y\}
\end{align}
which, if exists, is called a $\R$-minimizing solution of (\ref{1.1}). 

In practical applications, the exact data $y$ is in general not available, instead we only have a noisy data $y^\d$ 
satisfying 
\begin{align}\label{rbdgm.2}
\|y^\d - y\| \le \d, 
\end{align}
where $\d>0$ is the noise level. Due to the ill-posedness of the underlying problem, the solution of (\ref{R_min}) 
does not depend continuously on the data. How to use a noisy data $y^\d$ 
to stably reconstruct a $\R$-minimizing solution of (\ref{1.1}) is an important topic. 
To conquer the ill-posedness, many regularization methods have been developed to solve inverse problems; see \cite{bh12,BO2004,ehn96,FG2012,Jin2022,jw13,SLS2006,skhk12,zwt22,zwz22} and the references therein for instance. 

In this paper we will consider solving (\ref{R_min})  by a dual gradient flow. Throughout the paper we will assume that 
$\R$ is strongly convex. In case $\R$ is not strongly convex, we may consider its strong convex perturbation 
by adding it a small multiple of a strongly convex function; this does not affect much the reconstructed solution, 
see Proposition \ref{prop:A.1} in the appendix. The dual gradient flow for solving (\ref{R_min}) can be derived by 
applying the gradient flow to its dual problem. Since the Lagrangian function corresponding to 
(\ref{R_min}) with $y$ replaced by $y^\d$ is given by 
$$
L(x, \la) := \R(x) - \l \la, A x- y^\d\r, 
$$
we have the dual function 
$$
\inf_{x\in X} L(x, \la) = - \R^*(A^* \la) + \l \la, y^\d\r, \quad \forall \la \in Y,
$$
where $A^*: Y \to X^*$ denotes the adjoint of $A$ and $\R^*: X^* \to (-\infty, \infty]$ denotes the 
Legendre-Fenchel conjugate of $\R$. Thus the dual problem takes the form 
\begin{align}\label{DP}
\min_{\la\in Y} \left\{ d_{y^\d} (\la):= \R^*(A^*\la) - \l \la, y^\d\r \right\}. 
\end{align}
Since $\R$ is strongly convex, $\R^*$ is continuously differentiable over $X^*$ and its gradient $\nabla \R^*$ maps 
$X^*$ into $X$, see Proposition \ref{dgm.prop1}. Therefore, $\la \to d_{y^\d}(\la)$ is continuous differentiable with 
$\nabla d_{y^\d}(\la) = A \nabla \R^*(A^*\la) -y^\d$. Applying the gradient flow to solve (\ref{DP}) then gives 
$$
\frac{d}{dt} \la(t) = y^\d - A \nabla \R^*(A^* \la(t)), \quad t>0
$$
which can be equivalently written as 
\begin{equation}\label{asym0}
\begin{aligned}
x(t) & = \nabla \R^*(A^*\la(t)),\\
\frac{d}{dt} \la(t) & = y^\d-A x(t), \quad t>0
\end{aligned}
\end{equation}  
with a suitable initial value $\la(0)\in Y$. This is the dual gradient flow we will study 
for solving (\ref{R_min}). 

When $X$ is a Hilbert space and $\R(x) = \|x\|^2/2$, the dual gradient flow (\ref{asym0}) becomes the first 
order asymptotical regularization 
\begin{align}\label{AR}
\frac{d}{dt} x(t) = A^*(y^\d-Ax(t)),  \quad t>0
\end{align}
which is known as the Showalter's method. The analysis of (\ref{AR}) and its linear as well as nonlinear extensions 
in Hilbert spaces can be found in \cite{bdes21,lnw20,t94,zh19} for instance. Recently, higher order asymptotical 
regularization methods have also received attention for solving ill-posed problems in Hilbert spaces, 
see \cite{bdes21,zh19,zh20}. Due to the Hilbert space setting, the analysis of these asymptotical regularization 
methods can be performed by the powerful tool of spectral theory for bounded linear self-adjoint operators.  
Note that if the Euler scheme is used to discretize (\ref{AR}) one may obtain the linear Landweber iteration 
\begin{align}\label{Land0}
x_{n+1} = x_n - \gamma A^*(Ax_n-y^\d)
\end{align}
in Hilbert spaces with a suitable step-size $\gamma>0$. Therefore, (\ref{AR}) can be viewed as a continuous analogue 
of (\ref{Land0}) and the study of (\ref{AR}) can provide new insights about (\ref{Land0}). On the other hand,
by using other numerical schemes, we may produce from (\ref{AR}) iterative regularization methods far 
beyond the Landweber iteration. For instance, a class of order optimal iterative regularization methods have been proposed 
in \cite{r05} by discretizing (\ref{AR}) by the Runge-Kutta integrators; furthermore, by discretizing a second order 
asymptotical regularization method by a symplectic integrator -- the Str\"{o}rmer-Verlet method, a new order optimal 
iterative regularization method has been introduced in \cite{zh20} with acceleration effect. 

The dual gradient flow (\ref{asym0}) can be viewed as a continuous analogue of the Landweber iteration in Banach space 
\begin{equation}\label{dgm}
\begin{aligned}
 x_n & = \nabla \R^*(A^*\la_n), \\
 \la_{n+1} & = \la_n-\gamma(A x_n-y^\d)
\end{aligned}
\end{equation}
as (\ref{dgm}) can be derived from (\ref{asym0}) by applying the Euler discrete scheme. The method (\ref{dgm}) and 
its generalization to linear and nonlinear ill-posed problems in Banach spaces have been considered
in \cite{bh12,jw13,SLS2006} for instance. How to derive the convergence rates for (\ref{dgm}) has been a challenging question 
for a long time and it has been settled recently in \cite{Jin2022} by using some deep results from 
convex analysis in Banach spaces when the method is terminated by either an {\it a priori} stopping rule 
or the discrepancy principle. It should be mentioned that, by using other numerical integrators to 
discretize (\ref{asym0}) with respect to the time variable, one may produce new iterative regularization methods 
for solving (\ref{R_min}) in Banach sopaces that are different from (\ref{dgm}). 

In this paper we will analyze the convergence behavior of the dual gradient flow (\ref{asym0}). 
Due to the non-Hilbertian structure of the space $X$ and the non-quadraticity of the regularization 
functional $\R$, the tools for analyzing asymptotical regularization methods in Hilbert spaces are no 
longer applicable. The convergence analysis of (\ref{asym0}) is much more challenging and tools from 
convex analysis in Banach spaces should be exploited. Our analysis 
is inspired by the work in \cite{Jin2022}. We first prove some key properties on the dual 
gradient flow (\ref{asym0}) based on which we then consider its convergence behavior. Due to the 
propagation of noise along the flow, the primal trajectory $x(t)$ demonstrates the semi-convergence 
phenomenon: $x(t)$ approaches the sought solution at the beginning as the time $t$ increases; 
however, after a certain amount of time, the noise plays the dominated role and $x(t)$ begins 
to diverge from the sought solution. Therefore, in order to produce from $x(t)$ a reasonable 
approximate solution, the time $t$ should be chosen carefully. We consider several rules for 
choosing $t$. We first consider {\it a priori} parameter choice rules and establish convergence 
and convergence rate results. {\it A priori} rules can provide useful insights on the convergence 
property of the method. However, since it requires information on the unknown sought solution, 
the {\it a priori} parameter choice rules are of limited use in practice. To make the dual gradient 
flow (\ref{asym0}) more practical relevant, we then consider the choice of $t$ by {\it a posteriori} 
rules. When the noise level $\d$ is available, we consider choosing $t$ by the discrepancy principle 
and obtain the convergence and convergence rate results. In case the 
noise level information is not available or not reliable, the discrepancy principle may not be 
applicable; instead we consider the heuristic discrepancy principle which uses only the noisy data. 
Heuristic rules can not guarantee a convergence result in the worst case scenario according to the 
Bakushinskii's veto \cite{B1984}. However, under certain conditions on the noisy data, we can prove a convergence
result and derive some error estimates. 
Finally we provide numerical simulations to test the performance of the dual gradient flow.

\section{\bf Convergence analysis}

In this section we will analyze the dual gradient flow (\ref{asym0}). We will carry out the analysis
under the following conditions.

\begin{assumption}\label{dgm.ass1}
\begin{enumerate}[leftmargin = 0.8cm]
\item[\emph{(i)}] $X$ is a Banach space, $Y$ is a Hilbert space, and $A : X \to Y$ is a bounded linear operator;

\item[\emph{(ii)}] $\R : X\to (-\infty, \infty]$ is proper, lower semi-continuous, and strongly convex in the sense that there is a constant $c_0>0$ such that
$$
\R(\gamma x + (1-\gamma) x') + c_0 \gamma (1-\gamma) \|x-x'\|^2 \le \gamma \R(x) + (1-\gamma) \R(x')
$$
for all $x, x'\in \emph{dom}(\R)$ and $0\le \gamma \le 1$; moreover, each sublevel set of $\R$ is
weakly compact in $X$.

\item[\emph{(iii)}] The equation $Ax = y$ has a solution in $\emph{dom}(\R)$.
\end{enumerate}
\end{assumption}

For a proper convex function $\R: X \to (-\infty, \infty]$ we use $\p \R$ to denote its subdifferential, i.e.
$$
\p \R(x) := \{\xi \in X^*: \R(x')\ge \R(x) + \l \xi, x'-x\r, \, \forall x'\in X\}
$$
for all $x \in X$. It is known that $\p \R$ is a multi-valued monotone mapping.
Let $\mbox{dom}(\p \R) := \{x\in X: \p \R(x) \ne \emptyset\}$. If $\R$ is strongly convex as stated in
Assumption \ref{dgm.ass1} (ii), then
\begin{align}\label{sc1}
2 c_0 \|x'-x\|^2 \le \l \xi' -\xi, x'-x\r
\end{align}
for all $x', x \in \mbox{dom}(\p \R)$ with any $\xi'\in \p \R(x')$ and $\xi \in \p \R(x)$; moreover
$$
c_0 \|x'-x\|^2 \le D_{\R}^\xi (x', x)
$$
for all $x' \in X$, $x\in \mbox{dom}(\p \R)$ and $\xi \in \p \R(x)$, where
$$
D_\R^\xi(x', x) := \R(x') - \R(x) - \l \xi, x'-x\r
$$
denotes the Bregman distance induced by $\R$ at $x$ in the direction $\xi$.

For a proper, lower semi-continuous, convex function $\R: X \to (-\infty, \infty]$, its Legendre-Fenchel 
conjugate $\R^*: X^* \to (-\infty, \infty]$ is defined by
\begin{equation*}
\R^*(\xi):=\sup_{x\in X} \left\{\l\xi, x\r -\R(x)\right\}, \quad \xi\in X^*
\end{equation*}
which is also proper, lower semi-continuous, and convex and admits the duality property
\begin{equation}\label{FL1}
\xi\in \p \R(x) \Longleftrightarrow x\in \p \R^*(\xi) \Longleftrightarrow
\R(x) +\R^*(\xi) =\l \xi, x\r.
\end{equation}
If, in addition, $\R^*$ is strongly convex, then $\R^*$ has nice smoothness properties as stated in the
following result, see \cite[Corollary 3.5.11]{ZA02}.

\begin{proposition}\label{dgm.prop1}
Let $X$ be a Banach space and let $\R: X \to (-\infty, \infty]$ be a proper,
lower semi-continuous, strongly convex function as stated in Assumption \ref{dgm.ass1} (ii).
Then $\emph{dom}(\R^*) = X^*$, $\R^*$ is Fr\'{e}chet differentiable and its gradient
$\nabla \R^*$ maps $X^*$ into $X$ with the property
$$
\|\nabla \R^*(\xi') -\nabla \R^*(\xi) \| \le \frac{\|\xi'-\xi\|}{2c_0}
$$
for all $\xi', \xi \in X^*$.
\end{proposition}

Let Assumption \ref{dgm.ass1} hold. It is easy to show that (\ref{1.1}) has a unique $\R$-minimizing solution, 
which is denoted as $x^\dag$. We now consider the dual gradient flow (\ref{asym0}) to find approximation of $x^\dag$. 
By submitting $x(t) = \nabla \R^*(A^*\la(t))$ into the differential equation in (\ref{asym0}), we can see that
\begin{align}\label{ODE}
\frac{d}{dt} \la(t) = \Phi(\la(t)),
\end{align}
where
\begin{align}\label{Phi}
\Phi(\la) := y^\d - A \nabla \R^*(A^* \la)
\end{align}
for all $\la\in Y$. According to Proposition \ref{dgm.prop1}
we have
\begin{align}\label{8.3}
\|\Phi(\la') - \Phi(\la) \| & =\|A \nabla \R^*(A^* \la')- A \nabla \R^*(A^* \la) \| \nonumber \\
& \le \|A\| \|\nabla \R^*(A^* \la') -\nabla \R^* (A^*\la)\| \nonumber \\
& \le \frac{\|A\|}{2 c_0} \|A^* \la'-A^* \la\| \nonumber \\
& \le L \|\la'-\la\|
\end{align}
for all $\la', \la \in Y$, where $L := \|A\|^2/(2 c_0)$, i.e. $\Phi$ is globally Lipschitz continuous on $Y$.
Therefore, by the classical Cauchy-Lipschitz-Picard theorem, see \cite[Theorem 7.3]{B2011} for instance,
the differential equation (\ref{ODE}) with any initial value $\la(0)$ has a unique solution
$\la(t) \in C^1([0, \infty), Y)$. Defining $x(t) := \nabla \R^*(A^*\la(t))$ then shows that the dual
gradient flow (\ref{asym0}) with any initial value $\la(0)\in Y$ has a unique solution 
$(x(t), \la(t)) \in C([0, \infty), X) \times C^1([0, \infty), Y)$.

\subsection{Key properties of the dual gradient flow}

We will use the function $x(t)$ defined by the dual gradient flow (\ref{asym0}) with $\la(0)=0$ to approximate the
unique $\R$-minimizing solution $x^\dag$ of (\ref{1.1}) and consider the approximation property. 
Due to the appearance of noise in the data, $x(t)$ demonstrates the semi-convergence property, 
i.e. $x(t)$ tends to the sought solution at the beginning as the time $t$ increases, and after 
a certain amount of time, $x(t)$ diverges and the approximation property is deteriorated as $t$
continually increases. Therefore, it is necessary to determine a proper time at which the value 
of $x$ is used as an approximate solution. To this purpose, we first prove the
monotone property of the residual $\|A x(t) - y^\d\|$ which is crucial for designing 
parameter choice rules.

\begin{lemma}\label{lem.01}
Let Assumption \ref{dgm.ass1} hold. Then along the dual gradient flow (\ref{asym0}) the function $t \to \|A x(t) - y^\d\|$ is monotonically decreasing on $[0, \infty)$.
\end{lemma}

\begin{proof}
Let $0\le t_0 <t_1 <\infty$ be any two fixed numbers. We need to show
\begin{align}\label{8.1}
\|A x(t_1) - y^\d \| \le \|A x(t_0)-y^\d\|.
\end{align}
We achieve this by discretizing (\ref{asym0}) by the Euler method, showing the monotonicity holds
for the discrete method, and then using the approximation property of the discretization.

To this end, for any fixed integer $l\ge 1$ we set $h_l: = (t_1-t_0)/l$ and then define $\{x_k^{(l)}, \la_k^{(l)}\}_{k=0}^\infty$ by
\begin{align*}
x_k^{(l)}  = \nabla \R^*(A^* \la_k^{(l)}), \qquad
\la_{k+1}^{(l)}  = \la_k^{(l)} + h_l (y^\d - A x_k^{(l)})
\end{align*}
for $k = 0, 1, \cdots$, where $\la_0^{(l)} = \la(t_0)$ and hence $x_0^{(l)} = x(t_0)$.
By the definition of $x_k^{(l)}$ and (\ref{FL1}) we have $A^* \la_k^{(l)} \in \p \R(x_k^{(l)})$.
Since $\R$ is strongly convexity, we may use (\ref{sc1}) to obtain
\begin{align*}
2 c_0 \|x_{k+1}^{(l)} - x_k^{(l)}\|^2
& \le \left\l A^* \la_{k+1}^{(l)} - A^* \la_k^{(l)}, x_{k+1}^{(l)} - x_k^{(l)}\right\r \\
& = \left\l \la_{k+1}^{(l)} - \la_k^{(l)}, A(x_{k+1}^{(l)} - x_k^{(l)})\right\r.
\end{align*}
By using the definition of $\la_{k+1}^{(l)}$, we further have
\begin{align*}
2 c_0 \|x_{k+1}^{(l)} & - x_k^{(l)}\|^2
\le h_l \left\l y^\d - A x_k^{(l)}, A(x_{k+1}^{(l)} - x_k^{(l)})\right\r\\
& = \frac{h_l}{2} \left(\|A x_k^{(l)} - y^\d\|^2 - \|A x_{k+1}^{(l)} - y^\d\|^2 + \|A(x_{k+1}^{(l)} - x_k^{(l)})\|^2\right) \\
& \le \frac{h_l}{2} \left(\|A x_k^{(l)} - y^\d\|^2 - \|A x_{k+1}^{(l)} - y^\d\|^2\right)
+ \frac{h_l \|A\|^2}{2} \|x_{k+1}^{(l)} - x_k^{(l)}\|^2.
\end{align*}
By taking $l$ to be sufficiently large so that $h_l \|A\|^2 \le 4 c_0$, then we can obtain
\begin{align}\label{8.2}
\|A x_{k+1}^{(l)} - y^\d \| \le \|A x_k^{(l)} - y^\d\|, \quad k =0, 1, \cdots.
\end{align}
In particular, this implies that
$$
\|A x_l^{(l)} - y^\d\| \le \|A x_0^{(l)} - y^\d\| = \|A x(t_0) - y^\d\|.
$$
If we are able to show that $\|x_l^{(l)} - x(t_1)\| \to 0$ as $l \to \infty$, by taking $l \to \infty$
in the above inequality we can obtain (\ref{8.1}) immediately.

It therefore remains only to show $\|x_l^{(l)} - x(t_1)\| \to 0$ as $l \to \infty$. The argument is standard;
we include it here for completeness. Let $s_i = t_0 + i h_l$ for $i = 0, \cdots, l$. Using $\{x_k^{(l)}, \la_k^{(l)}\}$ we define $\la_l(t)$ for $t \in [t_0, t_1]$ as follows
\begin{align}\label{8.41}
\la_l(t) = \la_k^{(l)} + (t-s_k) (y^\d - A x_k^{(l)}), \quad \forall t \in [s_k, s_{k+1}].
\end{align}
Since $x_k^{(l)} = \nabla \R^*(A^* \la_k^{(l)})$, we have
$$
\la_l(t) = \la_k^{(l)} + (t-s_k) \Phi(\la_k^{(l)}), \quad \forall t \in [s_k, s_{k+1}],
$$
where $\Phi$ is defined by (\ref{Phi}). From the definition of $\la_l(t)$, it is easy to see
that $\la_l(s_k) = \la_k^{(l)}$. Furthermore,
for $t \in [t_0, t_1]$ we can find $0\le k \le l-1$ such that $t \in [s_k, s_{k+1}]$ and consequently
\begin{align*}
\la_l(t) & = \la_0^{(l)} + \sum_{i=0}^{k-1} (s_{i+1}-s_i) \Phi(\la_i^{(l)}) + (t-s_k) \Phi(\la_k^{(l)}) \\
& = \la(t_0) + \int_{s_0}^t \Phi(\la_l(s)) ds + \Delta_l(t),
\end{align*}
where
\begin{align*}
\Delta_l (t): = \sum_{i=0}^{k-1} \int_{s_i}^{s_{i+1}} \left[\Phi(\la_i^{(l)}) - \Phi(\la_l(s))\right] ds
+ \int_{s_k}^t \left[\Phi(\la_k^{(l)}) - \Phi(\la_l(s))\right] ds.
\end{align*}
Note that $\la(t) = \la(t_0) + \int_{t_0}^t \Phi(\la(s)) ds$. Therefore
\begin{align*}
\la_l(t) -\la(t)
& = \int_{s_0}^t \left[\Phi(\la_l(s)) - \Phi(\la(s))\right] ds + \Delta_l(t).
\end{align*}
Taking the norm on the both sides and using (\ref{8.3}) it follows that
\begin{align*}
\|\la_l(t) - \la(t)\| & \le L \int_{s_0}^t \|\la_l(s) - \la(s)\| ds + \|\Delta_l(t)\|.
\end{align*}
By using (\ref{8.3}) and (\ref{8.41}) we have
\begin{align*}
\|\Delta_l(t)\| & \le \sum_{i=0}^{k-1} L \int_{s_i}^{s_{i+1}} \|\la_i^{(l)} - \la_l(s)\| ds
+ L \int_{s_k}^t \|\la_k^{(l)} - \la_l(s)\| ds\\
& = \sum_{i=0}^{k-1} L \int_{s_i}^{s_{i+1}} (s-s_i)\| A x_i^{(l)} - y^\d\| ds
+ L \int_{s_k}^t (s-s_k) \|A x_k^{(l)} - y^\d\| ds.
\end{align*}
By using (\ref{8.2}), $s_0 = t_0$ and $t \in [s_k, s_{k+1}]$ with $0\le k \le l-1$ we thus obtain
\begin{align*}
\|\Delta_l(t)\| \le \frac{1}{2} L \|A x_0^{(l)} - y^\d\| \left( k h_l^2 + (t-s_k)^2\right) \le  M h_l,
\end{align*}
where $M := \frac{1}{2} L \|A x(t_0)- y^\d\| (t_1 -t_0)$. Therefore
\begin{align*}
\|\la_l(t) - \la(t)\| \le L \int_{t_0}^t \|\la_l(s) - \la(s)\| ds + M h_l
\end{align*}
for all $t \in [t_0, t_1]$. From the Gronwall inequality it then follows that
$$
\|\la_l(t) - \la(t)\| \le M h_l e^{L(t-t_0)}, \quad \forall t \in [t_0, t_1].
$$
Since $h_l \to 0$ as $l \to \infty$ and $\la_l(t_1) = \la_l(s_l) = \la_l^{(l)}$, we thus obtain from
the above equation that $\|\la_l^{(l)} - \la(t_1)\| \to 0$ as $l \to \infty$.
Since $x_l^{(l)} = \nabla \R^*(A^*\la_l^{(l)})$ and $x(t_1) = \nabla \R^*(A^*\la(t_1))$, by using the
continuity of $\nabla \R^*$ we can conclude $\|x_l^{(l)} - x(t_1)\| \to 0$ as $l \to \infty$. The proof
is therefore complete.
\end{proof}

Based on the monotonicity of $\|A x(t) - y^\d\|$ given in Lemma \ref{lem.01}, we can now prove
the following result which is crucial for the convergence analysis of the dual gradient flow (\ref{asym0}).

\begin{proposition}\label{prop2.0}
Let Assumption \ref{dgm.ass1} hold. Consider the dual gradient flow (\ref{asym0}) with $\la(0)=0$.
Then for any $\mu\in Y$ and $t>0$ there holds
\begin{equation}\label{Th0}
\frac{t}{2}\|Ax(t) -y^\d\|^2 +\frac{1}{2t}\left(\|\la(t)-\mu\|^2-\|\mu\|^2\right)
\le d_{\yd}(\mu) - d_{\yd}(\la(t)),
\end{equation}
where $d_{y^\d}(\mu) = \R^*(A^* \mu) - \l \mu, y^\d\r$ for any $\mu \in Y$.
\end{proposition}

\begin{proof}
According to the formulation of the dual gradient flow (\ref{asym0}), we have
\begin{align*}
\frac{d}{dt} \left[d_{y^\d}(\la(t))\right]
& = \left\l A \nabla \R^* (A^* \la(t)) - y^\d, \frac{d \la(t)}{dt}\right\r \\
& = - \|A x(t) - y^\d\|^2.
\end{align*}
Multiplying the both sides by $-t$ and then taking integration, we can obtain
\begin{align}\label{Th01}
\int_0^t s \|Ax(s) -y^\d\|^2 ds
& = - \int_0^t s \frac{d}{ds}\left[d_{y^\d}(\la(s))\right] ds \nonumber \\
& = -t d_{y^\d}(\la(t)) + \int_0^t d_{y^\d}(\la(s)) ds.
\end{align}
On the other hand, by the convexity of $d_{y^\d}$ we have
\begin{equation*}
d_{\yd}(\la) \le d_{\yd}(\mu) + \l \nabla d_{\yd}(\la), \la - \mu\r, \quad \forall \la, \mu \in Y.
\end{equation*}
Taking $\la = \la(t)$ and noting that $\frac{d \la(t)}{dt} = - \nabla d_{y^\d}(\la(t))$, we therefore have
\begin{equation*}
d_{\yd}(\la(t)) \le d_{\yd}(\mu) - \left\l \frac{d\lambda(t)}{dt}, \la(t) - \mu\right\r.
\end{equation*}
Integrating this equation then gives
\begin{align*}
\int_0^t  \left\l \frac{d\la(s)}{ds}, \la(s) - \mu\right\r ds
& \le \int_0^t \left(d_{\yd}(\mu) - d_{\yd}(\la(s))\right)ds \\
& = t d_{y^\d}(\mu) - \int_0^t d_{y^\d}(\la(s)) ds.
\end{align*}
By using $\lambda(0)=0$, we have
\begin{align*}
\int_0^t  \left\l \frac{d\la(s)}{ds}, \la(s) - \mu\right\r ds
& = \int_0^t \frac{d}{ds} \left(\frac{1}{2} \|\la(s) - \mu\|^2\right) ds \\
& = \frac{1}{2} \left(\|\la(t)-\mu\|^2 - \|\mu\|^2\right).
\end{align*}
Therefore
\begin{equation*}
\frac{1}{2}\left(\|\la(t)-\mu\|^2-\|\mu\|^2\right) \le t d_{\yd}(\mu) - \int_0^t d_{y^\d}(\la(s)) ds.
\end{equation*}
Adding this equation with \eqref{Th01} gives
\begin{align*}
\int_0^t s \|A x(s) - y^\d\|^2 ds + \frac{1}{2} \left(\|\la(t)-\mu\|^2 -\|\mu\|^2\right)
\le t \left(d_{y^\d}(\mu) - d_{\la^\d} (\la(t))\right).
\end{align*}
By virtue of the monotonicity of $t \to \|A x(t) - y^\d\|$ obtained in Lemma \ref{lem.01}, we have
$$
\int_0^t s \|A x(s)-y^\d\|^2 ds \ge \frac{t^2}{2} \|A x(t) - y^\d\|^2
$$
which, combining with the above equation, shows (\ref{Th0}).
\end{proof}

According to Proposition \ref{prop2.0}, we have the following consequences.

\begin{lemma}\label{lem:dgf2}
Let Assumption \ref{dgm.ass1} hold. Consider the dual gradient flow (\ref{asym0}) with $\la(0) =0$.
For any $\mu \in Y$ and $t>0$ there hold
\begin{align}
\frac{t}{2} \|A x(t) - y^\d\|^2
& \le d_y(\mu) - d_y(\la(t)) + \frac{\|\mu\|^2}{2t} + \frac{1}{2} t\d^2, \displaybreak[0] \label{eq:dgf3} \\
\frac{t}{2} \|A x(t)-y^\d\|^2 + \frac{1}{8t} \|\la(t)\|^2
&\le d_y(\mu) - d_y(\la(t)) + \frac{3 \|\mu\|^2}{4t} + t \d^2, \label{eq:dgf3.5}
\end{align}
where $d_y(\mu) = \R^*(A^*\mu) - \l \mu, y\r$ for any $\mu \in Y$.
\end{lemma}

\begin{proof}
According to Proposition \ref{prop2.0} and the relation $d_{y^\d}(\mu) - d_{y^\d}(\la(t))
= d_y(\mu) - d_y(\la(t)) + \l \la(t) -\mu , y^\d - y\r$, we can obtain
\begin{align}\label{eq:dgf1}
\frac{t}{2} \|A x(t) - y^\d\|^2
&\le d_y(\mu) - d_y(\la(t)) - \frac{\|\la(t)-\mu\|^2}{2t} + \frac{\|\mu\|^2}{2 t} \nonumber \displaybreak[0]\\
& \quad \, + \l \la(t) - \mu, y^\d - y\r.
\end{align}
By the Cauchy-Schwarz inequality and $\|y^\d - y\| \le \d$ we have
\begin{align}\label{eq:dgf3.1}
\l \la(t) - \mu, y^\d - y\r \le \d \|\la(t) -\mu \| \le \frac{\|\la(t) -\mu\|^2}{2t} + \frac{1}{2} t \d^2.
\end{align}
Combining this with (\ref{eq:dgf1}) shows (\ref{eq:dgf3}). In order to establish (\ref{eq:dgf3.5}),
slightly different from (\ref{eq:dgf3.1}) we now estimate the term $\l\la(t) - \mu, y^\d - y\r$ as
$$
\l \la(t) - \mu, y^\d - y\r \le \d \|\la(t) - \mu\| \le \frac{\|\la(t) -\mu\|^2}{4t} + t\d^2.
$$
It then follows from (\ref{eq:dgf1}) that
\begin{align*}
\frac{t}{2} \|A x(t) - y^\d\|^2 + \frac{\|\la(t)-\mu\|^2}{4t}
 \le d_y(\mu) - d_y(\la(t)) + \frac{\|\mu\|^2}{2 t} + t\d^2
\end{align*}
which together with the inequality
$
\|\la(t)\|^2 \le 2(\|\la(t) -\mu\|^2 + \|\mu\|^2)
$
then shows (\ref{eq:dgf3.5}).
\end{proof}

Next we will provide a consequence of Lemma \ref{lem:dgf2} which will be useful for deriving convergence rates
when the sought solution satisfies variational source conditons to be specified. To this end, we need the
following Fenchel-Rockafellar duality formula.

\begin{proposition}\label{prop1}
Let $X$ and $Y$ be Banach spaces, let $f: X \to (-\infty,\infty]$ and
$g: Y\to (-\infty,\infty]$ be proper convex functions, and let $A:X \to Y$
be a bounded linear operator. If there is $x_0\in \emph{dom}(f)$ such that
$Ax_0\in \emph{dom}(g)$ and $g$ is continuous at $Ax_0$, then
$$
\inf_{x \in X} \{f(x) + g(Ax)\} = \sup_{\eta\in Y^*}\{-f^*(A^*\eta) - g^*(-\eta)\}.
$$
\end{proposition}

The proof of Fenchel-Rockafellar duality formula can be found in various books on convex analysis,
see \cite[Theorem 4.4.3]{bz05} for instance.

\begin{lemma}\label{lem4.10}
Let Assumption \ref{dgm.ass1} hold. Consider the dual gradient flow (\ref{asym0}) with $\la(0)=0$.
For any $t>0$ there hold
\begin{align}
\frac{t}{2} \|Ax(t) -y^\d\|^2 &\le \eta(t) + \frac{t}{2}\d^2, \label{dd1.1}\\
\frac{t}{2} \|Ax(t) -y^\d\|^2 + \frac{1}{8t}\|\la(t)\|^2 &\le \eta(t) + t\d^2, \label{dd1.2}
\end{align}
where
\begin{equation}\label{eta}
\eta(t) := \sup_{x \in X}\left\{\R(x^\dag)-\R(x)- \frac{t}{3} \|Ax-y\|^2 \right\}.
\end{equation}
\end{lemma}

\begin{proof}
Since $y = A x^\dag$, from the Fenchel-Young inequality it follows that
$$
d_y(\la(t)) = \R^*(A^* \la(t)) - \l A^*\la(t), x^\dag\r \ge - \R(x^\dag).
$$
Combining this with the estimates in Lemma \ref{lem:dgf2} and noting that the estimates hold
for all $\mu\in Y$, we can obtain (\ref{dd1.1}) and (\ref{dd1.2}) with $\eta(t)$ defined by
$$
\eta(t) = \inf_{\mu\in Y} \left\{d_y(\mu) + \R(x^\dag) + \frac{3\|\mu\|^2}{4t}\right\}.
$$
It remains only to show that $\eta(t)$ can be given by (\ref{eta}). By rewriting $\eta(t)$ as
\begin{align*}
\eta(t) & = \R(x^\dag) - \sup_{\mu \in Y} \left\{-d_y(\mu) -  \frac{3\|\mu\|^2}{4t}\right\} \\
& = \R(x^\dag) - \sup_{\mu\in Y} \left\{-\R^*(A^* \mu) + \l \mu, y\r - \frac{3\|\mu\|^2}{4t}\right\},
\end{align*}
we may use Proposition \ref{prop1} to conclude the result immediately.
\end{proof}

\subsection{\bf A priori parameter choice}

In this subsection we analyze the dual gradient flow (\ref{asym0}) under {\it a priori} parameter
choice rule in which a number $t_\d>0$ is chosen depending only on the noise level $\d>0$ and possibly on the source
information of the sought solution $x^\dag$ and then $x(t_\d)$ is used as an approximate solution.
Although {\it a priori} parameter choice rule is of limited practical use, it provides valuable theoretical
guidance toward understanding the behavior of the flow. We first prove the following convergence result.

\begin{theorem}\label{dgf:thm1}
Let Assumption \ref{dgm.ass1} hold. Consider the dual gradient flow (\ref{asym0}) with $\la(0) = 0$.
If $t_\d>0$ is chosen such that $t_\d \to \infty$ and $\d^2 t_\d \to 0$ as $\d \to 0$, then
\begin{align}\label{eq:dgf08}
\R(x(t_\d)) \to \R(x^\dag) \quad \mbox{and} \quad D_\R^{\xi(t_\d)} (x^\dag, x(t_\d)) \to 0
\end{align}
as $\d \to 0$, where $\xi(t) := A^* \la(t)$. Consequently $\|x(t_\d) - x^\dag\| \to 0$ as $\d \to 0$.
\end{theorem}

\begin{proof}
Recall that $A^* \la(t_\d) \in \p \R(x(t_\d))$. It follows from the
convexity of $\R$ that
\begin{align}\label{eq:dgf10}
\R(x(t_\d)) & \le \R(x^\dag) + \l A^* \la(t_\d), x(t_\d) - x^\dag\r \nonumber \\
& = \R(x^\dag) + \l \la(t_\d), A x(t_\d) - y\r \nonumber \displaybreak[0]\\
& \le \R(x^\dag) + \|\la(t_\d)\| \|A x(t_\d) - y\|.
\end{align}
We need to treat the second term on the right hand side of (\ref{eq:dgf10}). We will
use (\ref{eq:dgf3.5}). By the Young-Fenchel inequality and $y = A x^\dag$, we have
\begin{align}\label{eq:dgf5}
d_y(\mu) = \R^*(A^* \mu) - \l A^* \mu, x^\dag\r \ge - \R(x^\dag)
\end{align}
which implies that
$$
\inf_{\mu \in Y} d_y(\mu) \ge - \R(x^\dag)>-\infty.
$$
Thus for any $\ep>0$ we can find $\mu_\ep\in Y$ such that
\begin{align}\label{eq:dgf5.5}
d_y(\mu_\ep) \le \inf_{\mu\in Y} d_y(\mu) + \ep.
\end{align}
By taking $t = t_\d$ and $\mu = \mu_\ep$ in (\ref{eq:dgf3.5}) we then obtain
\begin{align*}
\frac{t_\d}{2} \|A x(t_\d)-y^\d\|^2 + \frac{1}{8 t_\d} \|\la(t_\d)\|^2
\le \ep + \frac{3\|\mu_\ep\|^2}{4 t_\d} + t_\d \d^2.
\end{align*}
Using $\|A x(t_\d)-y\|^2 \le 2 (\|A x(t_\d) - y^\d\|^2 + \d^2)$ and the given conditions on $t_\d$,
it follows that
\begin{align*}
\limsup_{\d\to 0} \left(\frac{t_\d}{4} \|A x(t_\d)-y\|^2 + \frac{1}{8 t_\d} \|\la(t_\d)\|^2\right)
\le \ep.
\end{align*}
Sine $\ep>0$ can be arbitrarily small, we must have
\begin{align*}
t_\d \|A x(t_\d)-y\|^2 + \frac{1}{2 t_\d} \|\la(t_\d)\|^2\to 0 \quad \mbox{as } \d \to 0
\end{align*}
which in particular implies
\begin{align}\label{eq:dgf011}
\|A x(t_\d)-y\| \to 0 \quad \mbox{and} \quad \|\la(t_\d) \| \|A x(t_\d) - y\| \to 0
\end{align}
as $\d\to 0$. Thus, it follows from (\ref{eq:dgf10}) that
\begin{align}\label{eq:dgf101}
\limsup_{\d\to 0} \R(x(t_\d)) \le \R(x^\dag).
\end{align}
Based on (\ref{eq:dgf101}), we can find a constant $C$ independent of $\d$ such that
$\R(x(t_\d)) \le C$. Since every sublevel set of $\R$ is weakly compact, for any sequence
$\{y^{\d_k}\}$ of noisy data satisfying $\|y^{\d_k}-y\|\le \d_k \to 0$, by taking
a subsequence if necessary, we can conclude $x(t_{\d_k}) \rightharpoonup \bar x$ as
$k \to \infty$ for some $\bar x \in X$, where ``$\rightharpoonup$" denotes the weak convergence. 
By the weak lower semi-continuity of norms and $\R$
we can obtain from (\ref{eq:dgf011}) and (\ref{eq:dgf101}) that
$$
\|A \bar x - y\| \le \liminf_{k\to \infty} \|A x(t_{\d_k}) - y\| =0
$$
and
$$
\R(\bar x) \le \liminf_{k\to \infty} \R(x(t_{\d_k}) \le \limsup_{k\to \infty} \R(x(t_{\d_k}) \le \R(x^\dag).
$$
Thus $A \bar x = y$. Since $x^\dag$ is the $\R$-minimizing solution of $Ax = y$, we must have
$\R(\bar x) = \R(x^\dag)$ and therefore $\R(x(t_{\d_k})) \to \R(x^\dag)$ as $k \to \infty$.
By a subsequence-subsequence argument we can obtain $\R(x_{t_\d}) \to \R(x^\dag)$ as $\d\to 0$.
Consequently, by using (\ref{eq:dgf011}), we have
$$
D_\R^{\xi(t_\d)}(x^\dag, x(t_\d))
= \R(x^\dag) - \R(x(t_\d)) - \l \la(t_\d), y- A x(t_\d)\r \to 0
$$
as $\d \to 0$. The assertion $\|x(t_\d)-x^\dag\| \to 0$ then follows from the strong convexity of $\R$.
\end{proof}

Next we consider deriving the convergence rates under {\it a priori} parameter choice rule. For ill-posed
problems, convergence rate of the method depends crucially on the source condition of the sought solution $x^\dag$.
We now assume $x^\dag$ satisfies the following variational source condition.

\begin{assumption}\label{EM1}
For the unique $\R$-minimizing solution $x^\dag$ of (\ref{1.1}) there is an error measure function
$\E^\dag: X \to [0, \infty)$ satisfying $\E^\dag(x^\dag) =0$ such that
\[
\E^\dag(x) \le \R(x)-\R(x^\dag) + \omega\|Ax-y\|^q,
\]
for some $0<q\leq 1$ and some number $\omega>0$.
\end{assumption}

Since it was introduced in \cite{HKPS2007}, variational source condition has received tremendous attention,
various extensions, refinements and verification have been conducted and many convergence rate results
have been established based on this notion; see \cite{CJYZ2022,F2018,FG2012,HM2012,HW2015,HW2017,Jin2022}
for instance. In Assumption \ref{EM1}, the error measure function $\E^\dag$ is used to measure the speed of
convergence; it can be taken in various forms and the usual choice of $\E^\dag$ is the Bregman distance
induced by $\R$.

Under the variational source condition on $x^\dag$ specified in Assumption \ref{EM1}, the dual gradient flow
(\ref{asym0}) admits the following convergence rate when $t_\d$ is chosen by an {\it a priori} parameter
choice rule.

\begin{theorem}
Let Assumption \ref{dgm.ass1} hold. Consider the dual gradient flow (\ref{asym0}) with $\la(0) =0$
for solving (\ref{R_min}). If $x^\dag$ satisfies the variational source condition specified in Assumption
\ref{EM1}, then
$$
\E^\dag(x(t)) \le C \left(\omega^{\frac{2}{2-q}} t^{-\frac{q}{2-q}} + \omega^{\frac{1}{2-q}} t^{\frac{1-q}{2-q}} \d
+ \d^2 t + \omega \d^q\right)
$$
for all $t>0$, and consequently, for a
number $t_\d$ satisfying $t_\d \sim \omega \d^{q-2}$ there holds
$$
\E^\dag(x(t_\d)) \le  C \omega \d^q,
$$
where $C$ is a positive constant depending only on $q$.
\end{theorem}

\begin{proof}
By the variational source condition on $x^\dag$ we have
$$
\E^\dag(x(t)) \le \R(x(t)) - \R(x^\dag) + \omega \|A x(t) - y\|^q.
$$
Since $A^* \la(t) \in \p \R(x(t))$, it follows from the convexity of $\R$ that
\begin{align*}
\R(x(t))- \R(x^\dag) \le \l A^* \la(t), x(t) - x^\dag\r = \l \la(t), A x(t) -y\r.
\end{align*}
Therefore
\begin{align}\label{dgf.31}
\E^\dag(x(t)) & \le \l \la(t), A x(t)-y\r + \omega \|A x(t) -y\|^q \nonumber \\
& \le \|\la(t)\| \|A x(t) -y\| + \omega \|A x(t) - y\|^q.
\end{align}
We next use the inequality (\ref{dd1.2}) in Lemma \ref{lem4.10} to estimate $\|A x(t) -y\|$ and $\|\la(t)\|$.
By the variational source condition on $x^\dag$ and the nonnegativity of $\E^\dag$ we have $\R(x^\dag)-\R(x)
\le \omega\|A x-y\|^q$. Therefore
\begin{align}\label{eta:ee}
\eta(t) &\le \sup_{x \in X}\left\{\omega\|Ax-y\|^q- \frac{t}{3} \|Ax-y\|^2 \right\} \nonumber \\
& \le \sup_{s\ge 0} \left\{\omega s^q - \frac{t}{3}s^2 \right\} = c_2 \omega^{\frac{2}{2-q}} t^{-\frac{q}{2-q}}
\end{align}
with $c_2 = (1-\frac{q}{2})(\frac{3q}{2})^{\frac{q}{2-q}}$. This together with (\ref{dd1.2}) gives
$$
\frac{t}{2} \|A x(t) - y^\d\|^2 + \frac{1}{8t} \|\la(t)\|^2
\le c_2 \omega^{\frac{2}{2-q}} t^{-\frac{q}{2-q}} + \d^2 t
$$
which implies that
$$
\|A x(t) - y^\d\| \le C\left(\left(\frac{\omega}{t}\right)^{\frac{1}{2-q}} + \d\right)
\quad \mbox{ and } \quad
\|\la(t)\| \le C \left(\omega^{\frac{1}{2-q}} t^{\frac{1-q}{2-q}} + \d t\right).
$$
Hence we can conclude from (\ref{dgf.31}) that
\begin{align*}
\E^\dag(x(t)) & \le \|\la(t)\| \left(\|A x(t) - y^\d\| + \d\right) + \omega \left(\|A x(t) - y^\d\| + \d\right)^q \\
& \le C \left(\omega^{\frac{1}{2-q}} t^{\frac{1-q}{2-q}} + \d t\right)
\left(\left(\frac{\omega}{t}\right)^{\frac{1}{2-q}} + \d\right)
+ C \omega \left(\left(\frac{\omega}{t}\right)^{\frac{1}{2-q}} + \d\right)^q \\
& \le C \left(\omega^{\frac{2}{2-q}} t^{-\frac{q}{2-q}} + \omega^{\frac{1}{2-q}} t^{\frac{1-q}{2-q}} \d
+ \d^2 t + \omega \d^q\right).
\end{align*}
The proof is thus complete.
\end{proof}

\subsection{The discrepancy principle}

In order to obtain a good approximate solution, the {\it a priori} parameter choice rule considered 
in the previous subsection requires the knowledge of the unknown sought solution which limits 
its applications in practice. Assuming the availability of the noise level $\d>0$, we now consider 
the discrepancy principle which is an {\it a posteriori} rule that chooses a suitable time $t_\d$ 
based only on $\d$ and $y^\d$ such that the residual $\|A x(t_\d)-y^\d\|$ is roughly of the 
magnitude of $\d$. More precisely, the discrepancy principle can be formulated as follows. 

\begin{Rule}\label{Rule:DP}
Let $\tau>1$ be a given number. If $\|A x(0)-y^\d\| \le \tau \d$, we set $t_\d:=0$; otherwise, 
we define $t_\d >0$ to be a number such that 
$$
\|A x({t_\d})-y^\d\| = \tau \d.
$$
\end{Rule}

In the following result we will show that Rule \ref{Rule:DP} always outputs a finite number $t_\d$ 
and $x(t_\d)$ can be used as an approximate solution to $x^\dag$.

\begin{theorem}\label{thm:ALM1}
Let Assumption \ref{dgm.ass1} hold. Consider the dual gradient flow (\ref{asym0}). Then Rule \ref{Rule:DP} 
with $\tau>1$ outputs a finite number $t_\d$ and
\begin{align}\label{eq:dgf8}
\R(x(t_\d)) \to \R(x^\dag) \quad \mbox{and} \quad D_\R^{\xi(t_\d)} (x^\dag, x(t_\d)) \to 0
\end{align}
as $\d \to 0$, where $\xi(t) := A^* \la(t)$. Consequently $\|x(t_\d) - x^\dag\| \to 0$ as $\d \to 0$.
\end{theorem}

\begin{proof}
We first show that Rule \ref{Rule:DP} outputs a finite number $t_\d$.
To see this, let $t > 0$ be any number such that $\|A x(t) - y^\d\|\ge \tau \d$, from (\ref{eq:dgf3}) we
can obtain for any $\mu\in Y$ that
\begin{align}\label{eq:dgf4}
\frac{1}{2} (\tau^2-1) t \d^2 \le d_y(\mu) - d_y(\la(t))+ \frac{\|\mu\|^2}{2t}.
\end{align}
For any $\ep>0$, by taking $\mu$ to be the $\mu_\ep$ satisfying (\ref{eq:dgf5.5}),  we then obtain
\begin{align}\label{eq:dgf6}
\frac{1}{2} (\tau^2-1) t \d^2 \le \ep + \frac{\|\mu_\ep\|^2}{2 t}.
\end{align}
If Rule \ref{Rule:DP} does not output a finite number $t_\d$, then (\ref{eq:dgf6}) holds for
any $t>0$. By taking $t \to \infty$, we can see that the left hand side of (\ref{eq:dgf6}) 
goes to $\infty$ while the right hand side tends to $\ep$ which is a contradiction. Therefore 
Rule \ref{Rule:DP} must output a finite number $t_\d$.

We next show that $\d^2 t_\d \to 0$ as $\d\to 0$.  To see this, let $\{y^{\d_l}\}$ be any sequence
of noisy data satisfying $\|y^{\d_l}-y\| \le \d_l \to 0$ as $l\to \infty$.
If $\{t_{\d_l}\}$ is bounded, then there trivially holds $\d_l^2 t_{\d_l}\to 0$ as $l\to \infty$.
If $t_{\d_l}\to\infty$, then by taking $t = t_{\d_l}$ in (\ref{eq:dgf6}) we have
$$
0 \le \liminf_{l \to \infty} \d_l^2 t_{\d_l}\le \limsup_{l \to \infty} \d_l^2 t_{\d_l}
\le \frac{2\ep}{\tau^2-1}.
$$
Since $\ep>0$ can be arbitrarily small, we thus obtain $\lim_{l\to \infty} \d_l^2 t_{\d_l} =0$ again.

Now we are ready to show (\ref{eq:dgf8}). Since
\begin{align}\label{eq:dgf9}
\|A x(t_\d) - y\| \le \|A x(t_\d) - y^\d\| + \|y^\d - y\| \le (\tau+1) \d,
\end{align}
we immediately obtain
\begin{align}\label{eq:dgf90}
\|A x(t_\d) - y\| \to 0 \quad \mbox{ as } \d \to 0.
\end{align}
To establish $\lim_{\d\to 0} \R(x(t_\d)) = \R(x^\dag)$, we will use (\ref{eq:dgf10}).
We need to estimate $\|\la(t_\d)\|$. If $t_\d>0$, by taking $t = t_\d$ and $\mu = \mu_\ep$ in (\ref{eq:dgf3.5}), we have
\begin{align*}
\frac{\|\la(t_\d)\|^2}{8 t_\d}
\le \ep + \frac{3}{4 t_\d} \|\mu_\ep\|^2 + \d^2 t_\d.
\end{align*}
Therefore
$$
\|\la(t_\d)\|^2 \le 8\ep t_\d + 6 \|\mu_\ep\|^2 + 8 \d^2 t_\d^2
$$
which holds trivially when $t_\d =0$. Combining this with (\ref{eq:dgf9}) we obtain
$$
\|\la(t_\d)\|^2 \|A x(t_\d) - y\|^2
\le (\tau+1)^2 \left(8 \ep \d^2 t_\d + 6\|\mu_\ep\|^2 \d^2 + 8 \d^4 t_\d^2\right).
$$
With the help of the established fact $\d^2 t_\d\to 0$, we can see that
\begin{align}\label{eq:dgf11}
\lim_{\d\to 0} \|\la(t_\d) \| \|A x(t_\d) - y\| =0.
\end{align}
Thus, it follows from (\ref{eq:dgf10}) that
\begin{align}\label{eq:dgf111}
\limsup_{\d\to 0} \R(x(t_\d)) \le \R(x^\dag).
\end{align}
Based on (\ref{eq:dgf90}), (\ref{eq:dgf11}) and (\ref{eq:dgf111}), we can use the same argument
in the proof of Theorem \ref{dgf:thm1} to complete the proof.
\end{proof}

When the sought solution $x^\dag$ satisfies the variational source condition specified in Assumption \ref{EM1},
we can derive the convergence rate on $x(t_\d)$ for the number $t_\d$ determined by Rule \ref{Rule:DP}. 

\begin{theorem}
Let Assumption \ref{dgm.ass1} hold. Consider the dual gradient flow (\ref{asym0}) with $\la(0)=0$ for
solving (\ref{R_min}). Let $t_\d$ be a number determined by Rule \ref{Rule:DP} with $\tau>1$.
If $x^\dag$ satisfies the variational source condition specified in Assumption \ref{EM1}, then
$$
\E^\dag(x(t_\d)) \le  C \omega \delta^q,
$$
where $C$ is a positive constant depending only on $q$ and $\tau$.
\end{theorem}

\begin{proof}
By using the variational source condition on $x^\dag$, $A^* \la(t_\d) \in \p \R(x(t_\d))$, the convexity of $\R$
and the definition of $t_\d$,  we have
\begin{align}\label{dgf.16}
\E^\dag(x(t_\d))
& \le \R(x(t_\d))-\R(x^\dag) + \omega \|Ax(t_\d)-y\|^q \nonumber \\
& \le \l \la(t_\d), Ax(t_\d)-y\r +  \omega \|Ax(t_\d)-y\|^q \nonumber \\
& \le \|\la(t_\d)\|\|Ax(t_\d)-y\| +  \omega \|Ax(t_\d)-y\|^q \nonumber \\
& \le (1+\tau)\|\la(t_\d)\|\d +  \omega (1+\tau)^q\d^q.
\end{align}
If $t_\d=0$, then $\la(t_\d)=0$ and hence
$$
\E^\dag(x(t_\d)) \le \omega (1+\tau)^q\d^q.
$$
Therefore, we may assume $t_\d>0$. By using (\ref{dd1.1}) and (\ref{eta:ee}) with $t = t_\d$
and noting that $\|A x(t_\d) -y^\d\| = \tau \d$ we have
\begin{equation*}
\d^2 t_\d \le \frac{2}{\tau^2-1}\eta(t_\d)\le \frac{2c_2}{\tau^2-1} \omega^{\frac{2}{2-q}} t_\d^{-\frac{q}{2-q}},
\end{equation*}
which implies that
\begin{equation}\label{Tdelta}
t_\d \leq \left(\frac{2c_2}{\tau^2-1}\right)^{\frac{2-q}{2}} \omega \delta^{q-2}.
\end{equation}
Therefore, by using (\ref{dd1.2}) and (\ref{eta:ee}) with $t = t_\d$ we can obtain
\begin{align*}
\|\la(t_\d)\|^2 \le 8 t_\d \eta(t_\d) + 8 \d^2 t_\d^2
\le 8 c_2 \omega^{\frac{2}{2-q}} t_\d^{\frac{2(1-q)}{2-q}}+ 8 \d^2 t_{\d}^2
\le c_3 \omega^2 \d^{2(q-1)},
\end{align*}
where $c_3:= 4(1+\tau^2) (\frac{2c_2}{\tau^2-1})^{2-q}$. Combining this with (\ref{dgf.16}) we
finally obtain
\begin{equation*}
\E^\dag(x^\d(t_\d)) \le \left(\sqrt{c_3} (1+\tau) +  (1+\tau)^q\right) \omega \d^q.
\end{equation*}
The proof is complete.
\end{proof}

\subsection{\bf The heuristic discrepancy principle}

The performance of the discrepancy principle, i.e. Rule \ref{Rule:DP}, for the dual gradient flow
depends on the knowledge of the noise level. In case the accurate information on the noise
level is available, satisfactory approximate solutions can be produced, as shown in the
previous subsection. When noise level information is not available or reliable, we need to
consider heuristic rules, which use only $y^\d$, to produce approximate solutions \cite{hr96,j16,j17,KN2008,KR2020}. 
In this subsection we will consider the following heuristic discrepancy principle motivated by \cite{hr96,j17}.

\begin{Rule}\label{dgf.HR}
For a fixed number $a>0$ let $\Theta(t, y^\d) := (t+a)\|A x(t)-y^\d\|^2$ for $t \ge 0$
and choose $t_*:=t_*(y^\d)$ to be a number such that
\begin{align}\label{ALM:HR2}
\Theta(t_*, y^\d) = \min \left\{\Theta(t, y^\d): t\ge 0\right\}.
\end{align}
\end{Rule}

Note that the parameter $t_*$ chosen by Rule \ref{dgf.HR}, if exists, is independent of the noise level $\d>0$.
According to the Bakushinskii's veto, we can not expect a convergence result on $x(t_*)$ in the worst
case scenario as we obtained for the discrepancy principle in the previous subsection.
Additional conditions should be imposed on the noisy data in order to establish a convergence result on
this heuristic rule. We will use the following condition.

\begin{assumption}\label{dgf.ass2}
$\{y^\d\}$ is a family of noisy data with $\d:=\|y^\d-y\|\rightarrow 0$ and there is a constant
$\kappa>0$ such that
\begin{align}\label{eq:noise}
\| y^\d-A x\| \ge \kappa \|y^\d-y\|
\end{align}
for every $y^\d$ and every $x$ along the primal trajectory $S(y^\d)$ of the dual gradient flow, i.e. 
$S(y^\d):=\{x(t): t\ge 0\}$, where $x(t)$ is defined by the dual gradient flow (\ref{asym0}) 
using the noisy data $y^\d$.
\end{assumption}

This assumption was first introduced in \cite{j16}. It can be interpreted as follows. For ill-posed problems 
the operator $A$ usually has smoothing effect so that $A x$ admits certain regularity, while 
$y^\d$ contains random noise and hence exhibits salient irregularity. The condition (\ref{eq:noise}) 
roughly means that subtracting any regular function of the form $Ax$ with $x \in S(y^\d)$ can not significantly remove 
the randomness of noise. During numerical computation, we may testify Assumption \ref{dgf.ass2} by observing 
the tendency of $\|A x(t) - y^\d\|$ as $t$ increases: if $\|A x(t) - y^\d\|$ does not go below a very small number, 
we should have sufficient confidence that assumption \ref{dgf.ass2} holds true.

Under Assumption \ref{dgf.ass2} we have $\Theta(t, y^\d) \ge (\kappa \d)^2 t \to \infty$
as $t \rightarrow \infty$, which demonstrates that there must exist a finite integer $t_*$
satisfying (\ref{ALM:HR2}). We have the following convergence result on $x(t_*)$.

\begin{theorem}\label{thm4}
Let Assumption \ref{dgm.ass1} hold. Consider the dual gradient flow (\ref{asym0}) with $\la(0)=0$,
where $\{y^\d\}$ is a family of noisy data satisfying Assumption \ref{dgf.ass2}. Let $t_*:=t_*(y^\d)$
be determined by Rule \ref{dgf.HR}. Then 
$$
\R(x(t_*)) \to \R(x^\dag) \quad \mbox{and} \quad D_\R^{\xi(t_*)} (x^\dag, x(t_*)) \to 0
$$
as $\d \to 0$, where $\xi(t):= A^* \la(t)$. Consequently $\|x(t_*) - x^\dag\| \to 0$ as $\d \to 0$.
\end{theorem}

\begin{proof}
We first show that for $t_*:=t_*(y^\d)$ determined by Rule \ref{dgf.HR} there hold
\begin{equation}\label{eq:dgf16}
\Theta(t_*, y^\d) \to 0, \quad t_* \d^2 \to  0
\quad \mbox{and} \quad \|A x(t_*)-y^\d\| \to  0
\end{equation}
as $\d \to 0$. To see this, we set $\hat t_\d = 1/\d$ which satisfies $\hat t_\d\to \infty$
and $\hat t_\d \d^2 \to 0$ as $\d\rightarrow 0$. It then follows from the definition of $t_*$ that
\begin{align}\label{eq:dgf15}
\Theta(t_*, y^\d) \le \Theta(\hat t_\d, y^\d).
\end{align}
Let $\ep>0$ be arbitrarily small. As in the proof of Theorem \ref{dgf:thm1} we may choose $\mu_\ep\in Y$
such that (\ref{eq:dgf5.5}) hold. Therefore, it follows from (\ref{eq:dgf15}) and (\ref{eq:dgf3}) that
$$
\Theta(t_*, y^\d) \le \left(1+\frac{a}{\hat t_\d}\right)
\left(2\ep + \frac{\|\mu_\ep\|^2}{\hat t_\d} + \hat t_\d \d^2\right).
$$
By the property of $\hat t_\d$ we thus obtain
$$
\limsup_{\d\to 0} \Theta(t_*, y^\d)\le 2\ep.
$$
Since $\ep>0$ can be arbitrarily small, we must have $\Theta(t_*, y^\d) \to 0$ as $\d \to 0$.
Since $a>0$ and, by Assumption \ref{dgf.ass2}, $\|A x(t_*)-y^\d\| \ge \kappa\d$. We therefore have
$$
\|A x(t_*)-y^\d\|^2 \le \frac{\Theta(t_*, y^\d)}{a} \to 0 \quad \mbox{and} \quad
(\kappa\d)^2 t_* \le \Theta(t_*, y^\d) \to 0
$$
as $\d \to 0$. 

Next, by using (\ref{eq:dgf3.5}) and (\ref{eq:dgf5.5}), we have
$$
\|\la(t_*)\|^2 \le 8\ep t_* + 6\|\mu_\ep\|^2 + 8 t_*^2 \d^2
$$
which together with $a>0$ and $t_* \d^2 \to 0$ as $\d\to 0$ shows
$\|\la(t_*)|/\sqrt{t_*+a}\le C$ for some constant $C$ independent of $\d$.
Recall that $A^* \la(t_*)\in \p \R(x(t_*))$, we have
\begin{align*}
\R(x(t_*))& \le \R(x^\dag) + \l A^* \la(t_*), x(t_*)-x^\dag\r
= \R(x^\dag) + \l \la(t_*), A x(t_*)-y\r \displaybreak[0]\\
& \le \R(x^\dag) + \|\la(t_*)\| \left(\|A x(t_*)-y^\d\| +\d\right) \displaybreak[0]\\
& = \R(x^\dag) + \frac{\|\la(t_*)|}{\sqrt{t_*+a}} \left(\Theta(t_*, y^\d)^{1/2} + \d\sqrt{t_*+a} \right) \displaybreak[0]\\
& \le \R(x^\dag) + C \left(\Theta(t_*, y^\d)^{1/2} + \d \sqrt{t_*+a} \right).
\end{align*}
In view of (\ref{eq:dgf16}) we then obtain
\begin{align}\label{eq:dgf17}
\limsup_{\d\to 0} \R(x(t_*)) \le \R(x^\dag).
\end{align}
Based on (\ref{eq:dgf16}) and (\ref{eq:dgf17}), we can complete the proof by using the same argument
in the proof of Theorem \ref{dgf:thm1}.
\end{proof}

Next we provide an error estimate result on Rule \ref{dgf.HR} for the dual gradient flow (\ref{asym0})
under the variational source condition on $x^\dag$ specified in Assumption \ref{EM1}.

\begin{theorem}\label{thm:dgf3}
Let Assumption \ref{dgm.ass1} hold. Consider the dual gradient flow (\ref{asym0}) with $\la(0) =0$.
Let $t_*:=t_*(y^\d)$ be the number determined by Rule \ref{dgf.HR}. If $x^\dag$ satisfies the
variational source condition specified in Assumption \ref{EM1} and if $\d_*: = \|A x(t_*)-y^\d\|\ne 0$, then
\begin{align}\label{dgf.H1}
\E^\dag(x(t_*)) \le C_1 \left( 1 + \frac{\d^2}{\d_*^2}\right) \left(\d^2+\omega (\d+\d_*)^q\right),
\end{align}
where $C_1$ is a constant depending only on $a$ and $q$. If $\{y^\d\}$ is a family of noisy data satisfying
Assumption \ref{dgf.ass2}, then
\begin{align}\label{dgf.H2}
\E^\dag(x(t_*)) \le C_2  \left(\d^2+\omega (\d+\d_*)^q\right)
\end{align}
for some constant $C_2$ depending only on $a$, $q$ and $\kappa$.
\end{theorem}

\begin{proof}
Similar to the derivation of (\ref{dgf.16}) we have
\begin{align}\label{eq:dgf22}
\E^\dag(x(t_*)) \le \|\la(t_*)\| \|A x(t_*)-y\| + \omega \|A x(t_*)-y\|^q.
\end{align}
Based on this we will derive the desired estimate by considering the following two cases.

{\it Case 1: $t_* \le \omega/(\d+\d_*)^{2-q}$}. For this case we may use (\ref{dd1.2}) and
(\ref{eta:ee}) to derive that
$$
\|\la(t_*)\|^2  \le  8 c_2 \omega^{\frac{2}{2-q}} t_*^{\frac{2(1-q)}{2-q}} + 8 t_*^2 \d^2
\le 8(1+c_2) \omega^2 (\d+\d_*)^{2(q-1)}.
$$
Combining this with (\ref{eq:dgf22}) and using $\|A x(t_*)-y\| \le \d + \d_*$ we obtain
$$
\E^\dag(x(t_*))\le \left(1+ \sqrt{8(1+c_2)}\right) \omega \left(\d + \d_*\right)^q.
$$

{\it Case 2: $t_*> \omega/(\d+\d_*)^{2-q}$.} For this case we first use (\ref{eq:dgf22})
and the Cauchy-Schwarz inequality to derive that
\begin{align*}
\E^\dag(x(t_*))
& \le \omega (\d + \d_*)^q  + \frac{t_*}{2} \|A x(t_*)-y\|^2 + \frac{\|\la(t_*)\|^2}{2 t_*} \displaybreak[0]\\
& \le \omega (\d +\d_*)^q + t_* \d^2 + t_*\|A x(t_*)-y^\d\|^2+ \frac{\|\la(t_*)\|^2}{2 t_*}.
\end{align*}
With the help of (\ref{dd1.2}), (\ref{eta:ee}), the definition of $\Theta(t_*, y^\d)$
and  $t_*> \omega/(\d+\d_*)^{2-q}$ we then obtain
\begin{align}\label{eq:dgf20}
\E^\dag(x(t_*))
& \le \omega (\d +\d_*)^q + 4 c_2 \omega^{\frac{2}{2-q}} t_*^{-\frac{q}{2-q}}
+ \Theta(t_*, y^\d) + 5 t_* \d^2 \nonumber \displaybreak[0]\\
& \le (1+4c_2) \omega (\d + \d_*)^q + \left(1+\frac{5\d^2}{\d_*^2}\right) \Theta(t_*, y^\d).
\end{align}
By the definition of $t_*$ and the equation (\ref{dd1.1}) we have
\begin{align*}
\Theta(t_*, y^\d) \le \Theta(t, y^\d) \le 2c_2 \omega^{\frac{2}{2-q}} t^{-\frac{q}{2-q}} + t \d^2
+ a \left( 2 c_2 \omega^{\frac{2}{2-q}} t^{-\frac{2}{2-q}} +\d^2 \right)
\end{align*}
for all $t>0$. We now choose $t = \omega\d^{q-2}$. Then
\begin{align}\label{dgf.H5}
\Theta(t_*, y^\d)
\le \left(1+ 2c_2 \right) \omega \d^q + a(1+2c_2) \d^2.
\end{align}
Combining the above estimates on $\Theta(t_*, y^\d)$ with (\ref{eq:dgf20}) yields
\begin{align*}
\E^\dag(x(t_*))
& \le (1+4c_2) \omega (\d+\d_*)^q + (1+2c_2)\left(1+\frac{5\d^2}{\d_*^2}\right)
\left(a \d^2 + \omega \d^q\right).
\end{align*}
The shows (\ref{dgf.H1}).

Since Assumption \ref{dgf.ass2} implies $\d_*:=\|A x(t_*)-y^\d\| \ge \kappa \d$, (\ref{dgf.H2}) follows from (\ref{dgf.H1}) immediately.
\end{proof}

\begin{remark}
{\rm Under Assumption \ref{dgf.ass2} we can derive an estimate on $\E^\dag(x(t_*))$ independent of $\d_*$. 
Indeed, by definition we have $\d_*^2 = \frac{\Theta(t_*, y^\d)}{t_*+a} \le \frac{1}{a} \Theta(t_*, y^\d)$. 
By using (\ref{dgf.H5}) we then obtain $\d_*^2 = O(\d^q)$ which together with (\ref{dgf.H2}) gives 
$\E^\dag(x(t_*)) = O(\d^{q^2/2})$. In deriving this rate, we used $t_*\ge 0$ which is rather conservative.
The actual value of $t_*$ can be much larger as $\d\to 0$ and therefore better rate than $O(\d^{q^2/2})$
can be expected. 
}
\end{remark}

\section{\bf Numerical results}

In this section we present some numerical results to illustrate the numerical behavior of 
the dual gradient flow (\ref{asym0}) which is an autonomous ordinary differential equation. 
We discretize (\ref{asym0}) by the 4th-order Runge-Kutta method which takes the form (\cite{db02})
\begin{equation}\label{RK4}
\begin{aligned}
&\omega_{n,i}=\la_{n}+\Delta t \sum_{j=1}^{i-1} \gamma_{i,j}(y^\d - A k_{n,j}), \quad
  k_{n,i} = \nabla \R^*(A^*\omega_{n,i}), \quad  i=1,\cdots,4, \\
&\la_{n+1}=\la_{n}+\Delta t \sum_{i=1}^4 b_i(y^\d - A k_{n,i}), \quad
x_{n+1} = \nabla \R^*(A^*\la_{n+1})
\end{aligned}
\end{equation}
with suitable step size $\Delta t>0$, where $\Gamma:=(\gamma_{i,j})\in \mathbb{R}^{4\times 4}$ 
and $b := (b_i) \in \mathbb{R}^4$ are given by 
$$
\Gamma = \begin{pmatrix}
0   &  0  &  0   &  0  \\
1/2 &  0  &  0   &  0  \\
0   & 1/2 &  0   &  0  \\
0   &  0  &  1   &  0 \\
\end{pmatrix}
\quad \mbox{ and } \quad 
b = \begin{pmatrix}
1/6 \\
1/3 \\
1/3 \\
1/6
\end{pmatrix}.
$$
The implementation of (\ref{RK4}) requires determining $x := \nabla \R^*(A^* \la)$ for any $\la \in Y$, 
which, by virtue of (\ref{FL1}), is equivalent to solving the convex minimization problem
$$
x = \arg\min_{z\in X} \left\{ \R(z) - \l A^* \la, z\r \right\}.
$$
For many important choices of $\R$, $x := \nabla \R^*(A^* \la)$ can be given by an explicit formula; even if $x$ does 
not have an explicit formula, it can be determined by efficient algorithms.

\begin{example}\label{ex1}
{\rm 
Consider the first kind Fredholm integral equation of the form
\begin{equation}\label{Ax}
(Ax)(s) = \int_0^1 k(s,s')x(s')ds' = y(s) \quad  \mbox{ on }  [0,1],
\end{equation}
where the kernel $k$ is a continuous function on $[0,1]\times [0, 1]$. Clearly $A$ maps $L^1[0,1]$ into $L^2[0,1]$.
We assume the sought solution $x^\dag$ is a probability density function, i.e. $x^\dag \ge 0$ a.e. on $[0,1]$ 
and $\int_0^1 x^\dag(s) ds =1$. To find such a solution, we consider the dual gradient flow (\ref{asym0}) with 
\begin{align*} 
\R(x) : = f(x) + \iota_\Delta (x),
\end{align*}
where $\iota_\Delta$ denotes the indicator function of the closed convex set
\begin{align*}
\Delta := \left\{x\in L^1[0,1]: x\ge 0 \mbox{ a.e. on } [0,1] \mbox{ and } \int_0^1 x(s) ds =1\right\}
\end{align*}
in $L^1[0,1]$ and $f$ denotes the negative of the Boltzmann-Shannon entropy, i.e.
$$
f(x) := \left\{\begin{array}{lll}
\int_0^1 x(s) \log x(s) ds & \mbox{ if } x \in L_+^1[0,1] \mbox{ and } x \log x \in L^1[0,1], \\[1.2ex]
\infty & \mbox{ otherwise},
\end{array}\right.
$$
where $L_+^1[0,1]:= \{x \in L^1[0,1]: x \ge 0 \mbox{  a.e. on } [0,1]\}$. According to \cite{AH1991,BL1991,E1993,EL1993,Jin2022}, 
$\R$ satisfies Assumption \ref{dgm.ass1}. By the Karush-Kuhn-Tucker theory, for any $\la\in L^2[0,1]$,  
$x:=\nabla \R^*(A^*\la)$ is given by $x := e^{A^*\la}/\int_0^1 e^{(A^*\la)(s)} ds$.

For numerical simulations we consider (\ref{Ax}) with $k(s,s') = 4e^{-(s-s')^2/0.0064}$
and assume the sought solution is 
$$
x^\dag (s) = c\left( e^{-60 (s-0.3)^2}+0.3e^{-40(s-0.7)^2}\right),
$$
where the constant $c>0$ is chosen to ensure $\int_0^1 x(s) ds = 1$ so that $x^\dag$ is a probability 
density function.
In the numerical computation we divide $[0, 1]$ into $m=800$ subintervals of equal length 
and approximate integrals by the trapezoidal rule. We add random Gaussian noise to the exact data 
$y:=Ax^\dag$ to obtain the noisy data $y^\d$ whose noise level is $\d:=\|y-y^\d\|_{L^2}$. 
With $y^\d$ we reconstruct the sought solution $x^\dag$ by using the dual gradient flow 
(\ref{asym0}) with $\la(0)=0$ which is solved approximately by the 4th-order Runge-Kutta method 
as described in (\ref{RK4})
with $\Delta t = 0.4$. We consider the choice of the proper time $t$ by the discrepancy principle, i.e. 
Rule \ref{Rule:DP} and the heuristic discrepancy principle, i.e. Rule \ref{dgf.HR}.

\begin{table}[http] 
\caption{Numerical results for Example \ref{ex1} by the dual gradient flow (\ref{asym0}) under 
Rule \ref{Rule:DP} and Rule \ref{dgf.HR}.}\label{table1}
\begin{center}
\begin{tabular}{cccccccc}
\hline
&\multicolumn{2}{c} {Rule \ref{Rule:DP} with $\tau=1.1$ }& \multicolumn{2}{c}{Rule \ref{Rule:DP} with $\tau=6.0$} 
& \multicolumn{2}{c}{Rule \ref{dgf.HR} with $a = 0.1$}
 \\
 \cline{2-7}
$\d $  & $t$ & $RE$ & $t$ & $RE$ &  $t$ & $RE$ \\
\hline
1e-1 & 13.6    & 1.1514e-1 & 0.8    & 6.6167e-1  & 7.2    & 1.9012e-1  \\
1e-2 & 132.8   & 3.2177e-2 & 10.4   & 1.3664e-1  & 63.2   & 3.8761e-2  \\
1e-3 & 1810.4  & 1.2145e-2 & 106    & 3.1348e-2  & 680.8  & 1.4750e-2 \\
1e-4 & 26532.4 & 5.0245e-3 & 1376.4 & 1.2854e-2  & 8001.2 & 8.2320e-3  \\
\hline
\end{tabular}
\end{center}
\end{table}

To demonstrate the numerical performance of (\ref{asym0}) quantitatively, we perform the numerical 
simulations using noisy data with various noise levels. In Table \ref{table1} we report the time 
$t$ determined by Rule \ref{Rule:DP} with $\tau = 1.1$ and $\tau = 6.0$ and by Rule \ref{dgf.HR} with 
$a = 0.1$, we also report the corresponding relative errors $RE:=\|x(t)-x^\dag\|_{L^1}/\|x^\dag\|_{L^1}$. 
The values $\tau = 1.1$ and $\tau = 6.0$ used in Rule \ref{Rule:DP} correspond to the proper estimation and 
overestimation of noise levels respectively. 
The results in Table \ref{table1} show that Rule \ref{Rule:DP} with a proper estimate of the noise level 
can produce nice reconstruction results; in case the noise level is overestimated, less accurate reconstruction 
results can be produced. Although Rule \ref{dgf.HR} does not utilize any information on noise level, it 
can produce satisfactory reconstruction results. 

\begin{figure}[htpb]
\centering
\includegraphics[width = 0.32\textwidth]{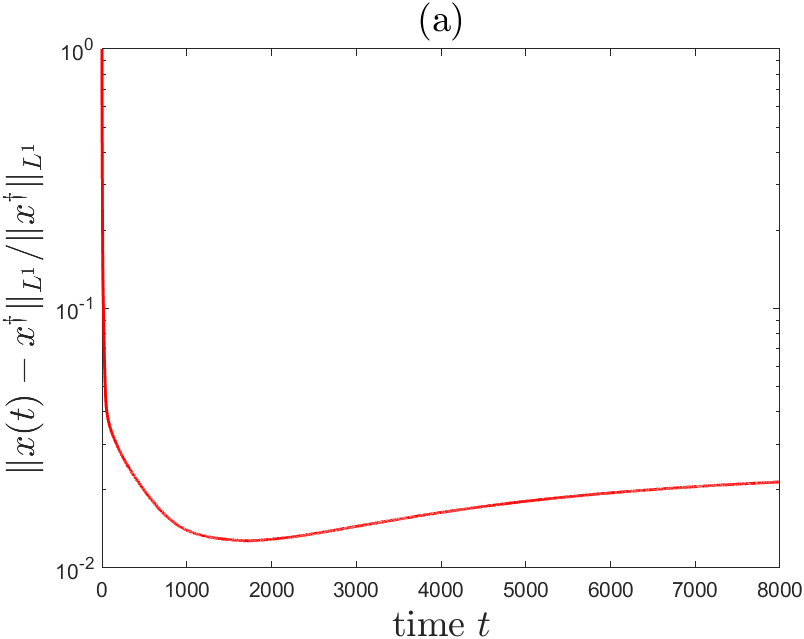}
\includegraphics[width = 0.32\textwidth]{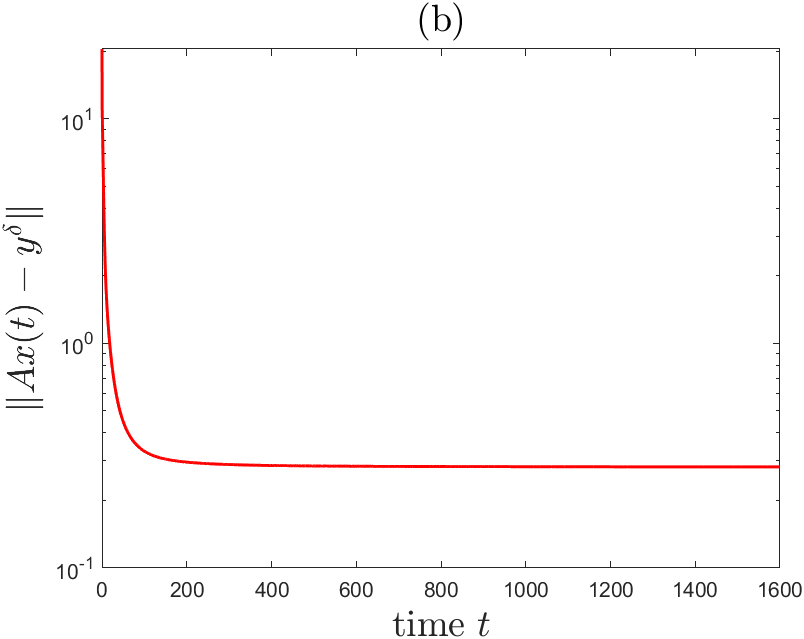}
\includegraphics[width = 0.32\textwidth]{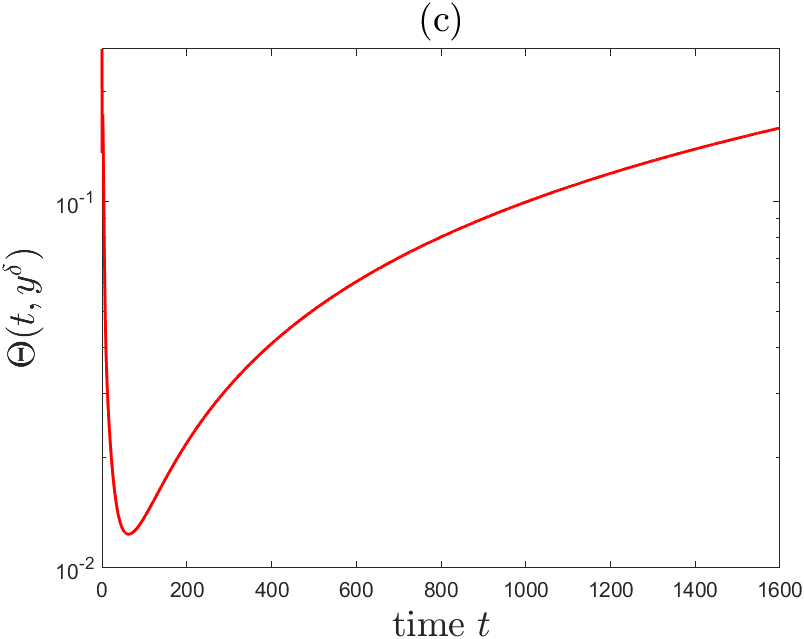}
\vskip 0.15cm
\includegraphics[width = 0.32\textwidth]{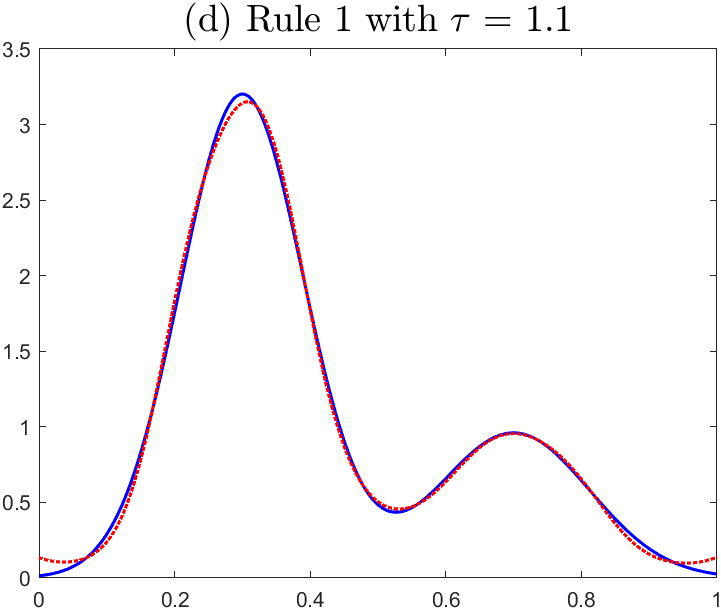}
\includegraphics[width = 0.32\textwidth]{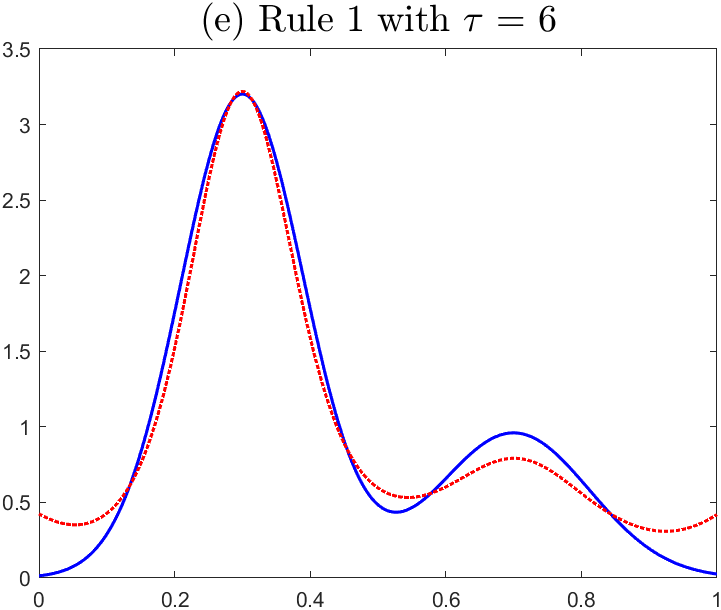}
\includegraphics[width = 0.32\textwidth]{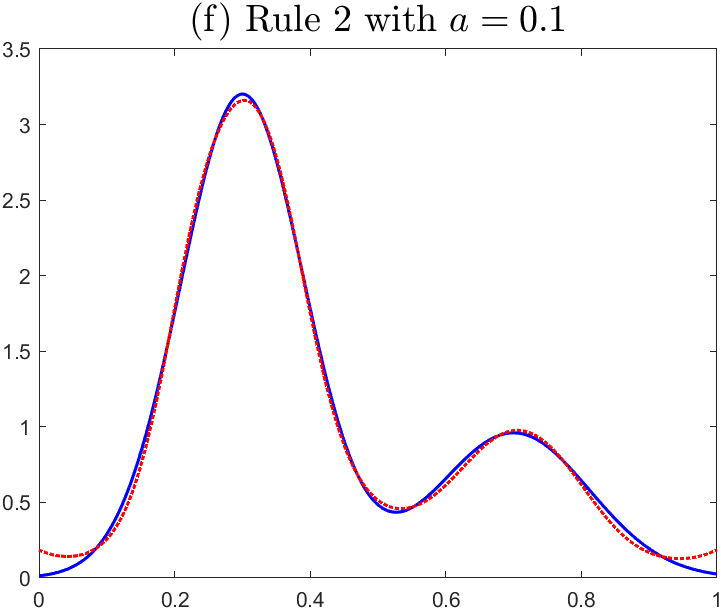}
\caption{The reconstructed results for Example \ref{ex1} using noisy data with noise level $\d = 0.01$. }
\label{fig1:pdf}
\end{figure}

In order to visualize the performance of (\ref{asym0}), in Figure \ref{fig1:pdf} we present the
reconstruction results using a noisy data $y^\d$ with the noise level $\d = 0.01$. Figure \ref{fig1:pdf} (a) 
plots the relative error versus the time which indicates the dual gradient flow (\ref{asym0}) possesses 
the semi-convergence property. Figure \ref{fig1:pdf} plots the residual $\|A x(t) - y^\d\|_{L^2}$ versus the time 
$t$ which demonstrates that Assumption \ref{dgf.ass2} holds with sufficient confidence if the data is corrupted by 
random noise. Figure \ref{fig1:pdf} (c) plots $\Theta(t, y^\d)$ versus $t$ which clearly shows that $t\to \Theta(t, y^\d)$ 
achieves its minimum at a finite number $t>0$. Figure \ref{fig1:pdf} (d)--(f) present the respective 
reconstruction results $x(t)$ with $t$ chosen by Rule \ref{Rule:DP} with $\tau = 1.1$ and $\tau = 6.0$ and by 
Rule \ref{dgf.HR} with $a = 0.1$. 
}
\end{example}

\begin{example}\label{ex2}
{\rm
We next consider the computed tomography which consists in determining the density of cross sections 
of a human body by measuring the attenuation of X-rays as they propagate through the biological tissues.
In the numerical simulations we consider  test problems that model the standard 2D parallel-beam tomography.  
We use the full angle model with 45 projection angles evenly distributed between 1 and 180 degrees, 
with 367 lines per projection. Assuming the sought image is discretized on a $256 \times 256$ pixel grid, 
we may use the function \texttt{paralleltomo} in the MATLAB package AIR TOOLS \cite{Hansen2012} to discretize 
the problem. It leads to an ill-conditioned linear algebraic system $Ax = y$,  
where $A$ is a sparse matrix of size $M\times N$, where $M=16515$ and $N=66536$.

Let the true image be the modified Shepp-Logan phantom of size $256\times 256$ generated by MATLAB.
Let $x^{\dag }$ denote the vector formed by stacking all the columns of the true image and let $y = A x^\dag$ 
be the true data. We add Gaussian noise on $y$ to generate a noisy data $y^\d$ with relative noise level
$\d_{rel}=\|y^\d-y\|_2/\|y\|_2$ so that the noise level is $\d = \d_{rel} \|y\|_2$. We will use $y^\d$ 
to reconstruct $x^\dag$. In order to capture the feature of the sought image, we take
\begin{align*} 
\R(x) = \frac{1}{2\beta} \|x\|_2^2 + |x|_{TV}
 \end{align*}
with a constant $\beta>0$, where $|x|_{TV}$ denotes the total variation of $x$.
This $\R$ is strongly convex with $c_0 = \frac{1}{2\beta}$.
The determination of $x= \nabla \R^*(A^* \la)$ for any given $\la$ is equivalent to solving the 
total variation denoising problem
\begin{align*} 
x = {\rm arg}\min_{z}\left\{\frac{1}{2\beta}\|z-\beta A^*\la\|_2^2 + |z|_{TV}\right\} 
\end{align*}
which can be solved efficiently by the primal dual hybrid gradient method \cite{BR2012,ZC2008};
we use $\beta=1$ in our computation. With a noisy data $y^\d$ we reconstruct the true image by the dual 
gradient method (\ref{asym0}) with $\la(0) = 0$ which is solved approximately by the 4th Runge-Kutta method (\ref{RK4}) 
with constant step-size $\Delta t = 0.4\times 10^{-3}$.

\begin{table}[http] 
\caption{Numerical results for Example \ref{ex2} by the dual gradient flow (\ref{asym0}) under 
Rule \ref{Rule:DP} and Rule \ref{dgf.HR}.} \label{table2}
\begin{center}
\begin{tabular}{cccccccc}
\hline
&\multicolumn{2}{c} {Rule \ref{Rule:DP} with $\tau=1.05$}& \multicolumn{2}{c}{Rule \ref{Rule:DP} with $\tau=3.0$} 
& \multicolumn{2}{c}{Rule \ref{dgf.HR} with $a = 0.1$}\\
\cline{2-7}
    $\delta_{rel}$  & $t$ & $RE$ & $t$ & $RE$ &  $t$ & $RE$ \\
\hline
5e-2 & 0.0260 & 1.9677e-1  & 0.0144 & 3.0193e-1 & 0.0324 & 1.7331e-1  \\
1e-2 & 0.1420 & 6.5240e-2  & 0.0212 & 2.2035e-1 & 0.0992 & 8.1327e-2  \\
5e-3 & 0.2640 & 3.6406e-2  & 0.0380 & 1.4581e-1 & 0.2080 & 4.5257e-2 \\
1e-3 & 1.0136 & 7.3525e-3  & 0.2160 & 4.1178e-2 & 0.6392 & 1.1537e-2  \\
5e-4 & 1.7620 & 3.6169e-3  & 0.5028 & 2.1548e-2 & 1.5904 & 6.0784e-3  \\
\hline
\end{tabular}
\end{center}
\end{table}

We demonstrate in Table \ref{table2} the performance of (\ref{asym0}) quantitatively by performing the numerical 
simulations using noisy data with various relative noise levels. We report the value of $t$ determined by Rule \ref{Rule:DP} 
with $\tau = 1.05$ (proper estimation case) and $\tau = 3.0$ (overestimation case) and by Rule \ref{dgf.HR} with 
$a = 0.1$, we also report the corresponding relative errors $RE:=\|x(t)-x^\dag\|_2/\|x^\dag\|_2$. 
The results show that Rule \ref{Rule:DP} with a proper estimate of the noise level 
can produce satisfactory reconstruction results; overestimation of noise level can deteriorate the reconstruction 
accuracy.  Rule \ref{dgf.HR} can produce nice reconstruction results even if it does not rely on the information 
of noise level. 

\begin{figure}[htpb]
\centering
\includegraphics[width = 0.32\textwidth]{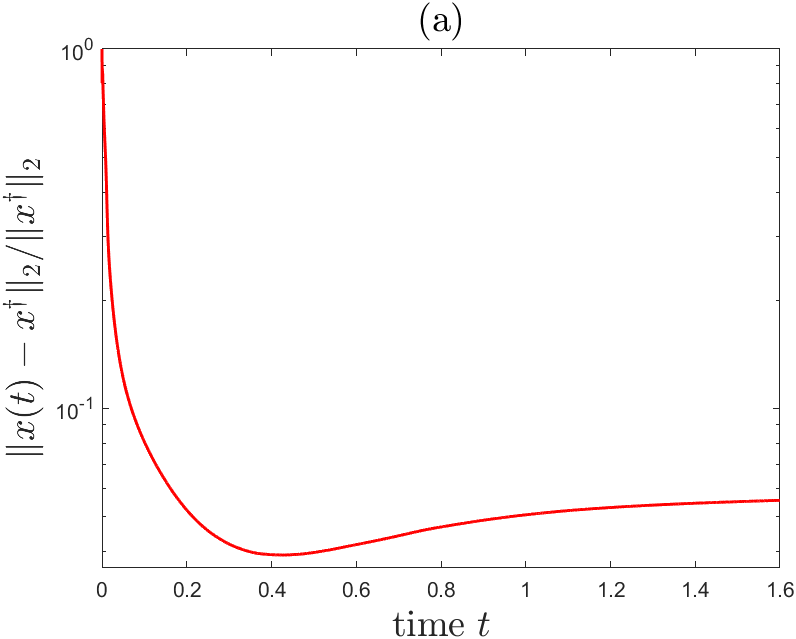}
\includegraphics[width = 0.32\textwidth]{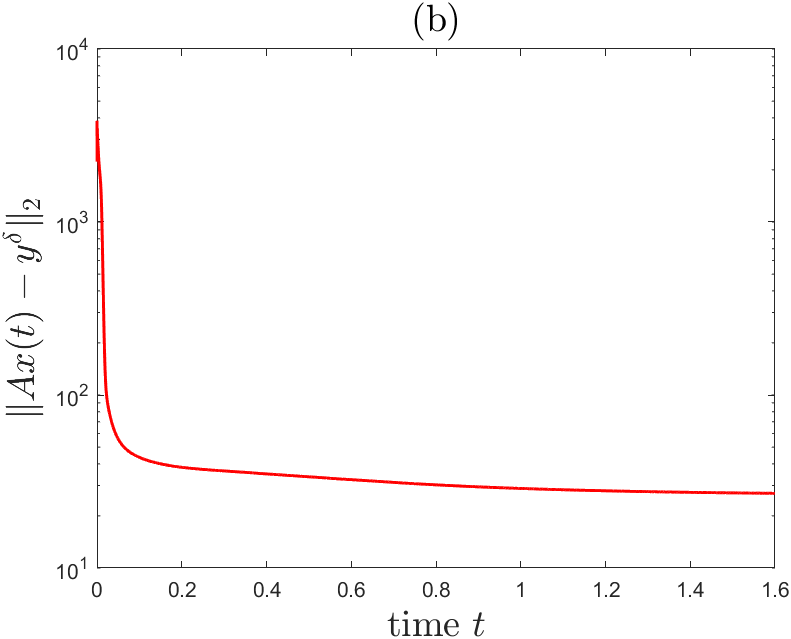}
\includegraphics[width = 0.32\textwidth]{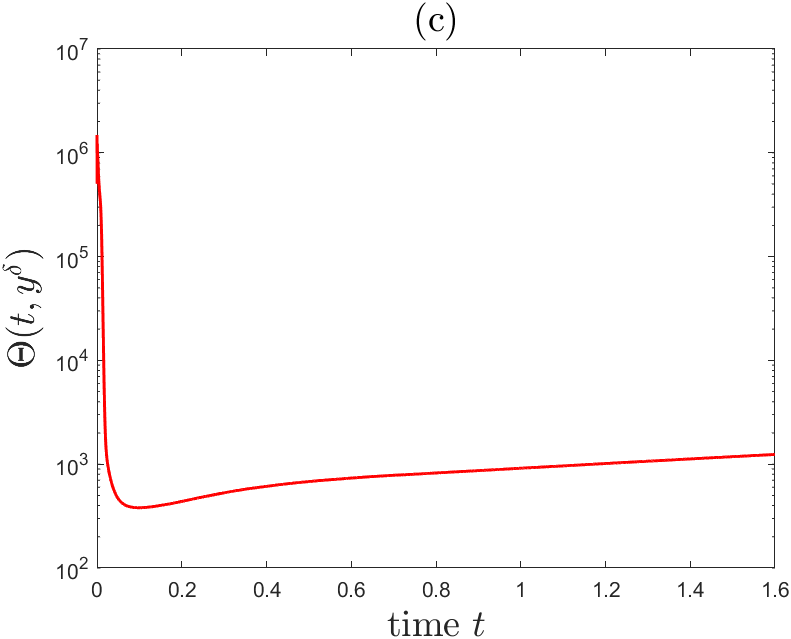}
\vskip 0.15cm
\includegraphics[width = 0.23\textwidth]{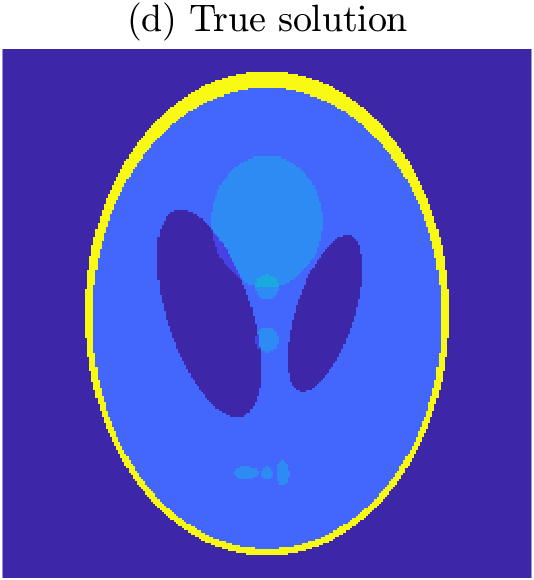}
\includegraphics[width = 0.23\textwidth]{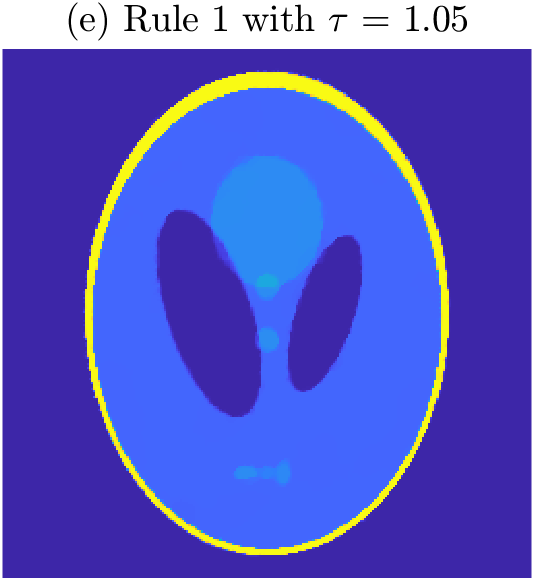}
\includegraphics[width = 0.23\textwidth]{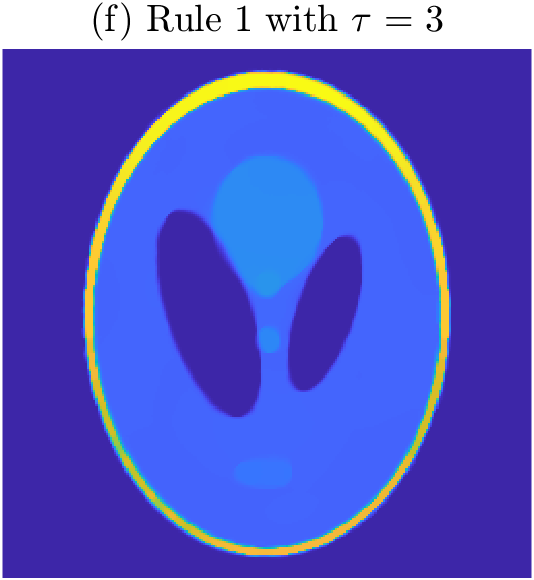}
\includegraphics[width = 0.23\textwidth]{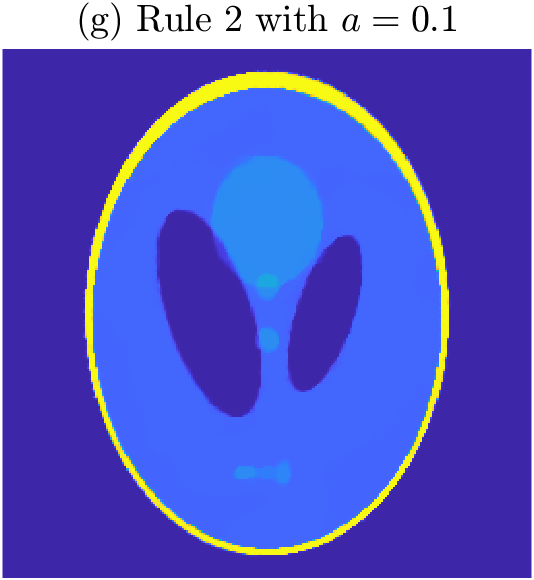}
\caption{The reconstructed results for Example \ref{ex2} using noisy data with relative noise level $\d_{rel} = 0.01$.}
\label{fig2:CT}
\end{figure}

To visualize the performance of (\ref{asym0}), in Figure \ref{fig2:CT} we present the
reconstruction results using a noisy data $y^\d$ with the relative noise level $\d_{rel} = 0.01$. 
Figure \ref{fig2:CT} (a)-(c) plot the relative errors, the residual $\|A x(t) - y^\d\|_2$ and 
the value of $\Theta(t, y^\d)$ versus the time $t$. These results demonstrate the semi-convergence phenomenon 
of (\ref{asym0}), the validity of Assumption \ref{dgf.ass2} and well definedness of Rule \ref{dgf.HR}. 
Figure \ref{fig2:CT} (d)-(g) plot the true image, the reconstructed images by Rule \ref{Rule:DP} with 
$\tau = 1.05$ and $\tau = 3.0$, and the reconstruction by Rule \ref{dgf.HR} with $a = 0.1$. 
}
\end{example}

\section*{\bf Acknowledgements}
The work of Q Jin is partially supported by the Future Fellowship of the Australian Research Council (FT170100231) 
and the work of W Wang is supported by the National Natural Science Foundation of China (No. 12071184).

\appendix
\section{}

Consider (\ref{R_min}), we want to show that by adding a small multiple of a strongly convex function 
to $\R$ does not affect the solution too much. Actually, we can prove the following result which is an 
extension of \cite[Theorem 4.4]{COS2009} that proves a special instance related to sparse recovery in 
finite-dimensional Hilbert spaces.

\begin{proposition}\label{prop:A.1}
Let $X$ and $Y$ be Banach space, $A: X \to Y$ a bounded linear operator and $y \in \emph{Ran}(A)$, 
the range of $A$. Let $\R: X \to (-\infty, \infty]$ be a proper, lower semi-continuous, convex 
function with $S\cap \emph{dom}(\R) \ne \emptyset$, where  
$$
S := \arg\min\left\{\R(x): x\in X \mbox{ and } A x = y\right\}.
$$
Assume that $\Psi: X \to (-\infty, \infty]$ is a proper, lower semi-continuous, strongly convex function
with $\emph{dom}(\R) \subset \emph{dom}(\Psi)$ such that every sublevel set of $\R+ \Psi$ is weakly compact. 
For any $\a>0$ define
$$
x_\a := \arg\min\left\{\R(x) + \a \Psi(x): x\in X \mbox{ and } A x = y\right\}.
$$
Then $x_\a \to x^*$ and $\R(x_\a) \to \R(x^*)$ as $\a \to 0$, where $x^*\in S$ is such
that $\Psi(x^*) \le \Psi(x)$ for all $x \in S$.
\end{proposition}

\begin{proof}
Since $S$ is convex and $\Psi$ is strongly convex, $x_\a$ and $x^*$ are uniquely defined. By 
the definition of $x_\a$ and $x^*$ we have
\begin{align}\label{A.1}
\R(x_\a) + \a \Psi(x_\a) \le \R(x^*) + \a \Psi(x^*) \tag{A.1}
\end{align}
and
\begin{align}\label{A.2}
\R(x^*) \le \R(x_\a). \tag{A.2}
\end{align}
Combining these two equations gives
\begin{align}\label{A.3}
\Psi(x_\a) \le \Psi(x^*) < \infty. \tag{A.3}
\end{align}
Thus, it follows from (\ref{A.1}) that $\limsup_{\a\to 0} \R(x_\a) \le \R(x^*)$ which
together with (\ref{A.2}) shows
\begin{align}\label{A.4}
\lim_{\a\to 0} \R(x_\a) = \R(x^*). \tag{A.4}
\end{align}

Next we show $x_\a \to x^*$ as $\a\to 0$. Since every sublevel set of $\R+\Psi$ is weakly compact, 
for any sequence $\{\a_k\}$ with $\a_k \to 0$, by taking a subsequence if necessary, we may 
use (\ref{A.3}) and (\ref{A.4}) to conclude $x_{\a_k} \rightharpoonup \hat x$ as $k\to \infty$
for some element $\hat x \in X$. Since $A x_{\a_k} = y$ and $A$ is bounded, letting $k \to \infty$ gives
$A \hat x = y$. By the weak lower semi-continuity of $\R$ and (\ref{A.4}) we also have
$$
\R(\hat x) \le \liminf_{k\to \infty} \R(x_{\a_k}) = \R(x^*).
$$
Since $x^* \in S$, we thus have $\hat x \in S$. Consequently $\Psi(x^*) \le \Psi(\hat x)$ by the definition
of $x^*$. By (\ref{A.3}) and the weak lower semi-continuity of $\Psi$ we also have
\begin{align}\label{A.5}
\Psi(\hat x) \le \liminf_{k\to \infty} \Psi(x_{\a_k}) \le \limsup_{k\to \infty} \Psi(x_{\a_k}) \le \Psi(x^*).
\tag{A.5}
\end{align}
Thus $\hat x, x^* \in S$ and $\Psi(\hat x) =\Psi(x^*)$. By uniqueness we must have $\hat x = x^*$. 
Consequently $x_{\a_k}\rightharpoonup x^*$ and, by (\ref{A.5}), $\Psi(x_{\a_k}) \to \Psi(x^*)$ 
as $k \to \infty$. Since $\Psi$ is strongly convex, it admits the Kadec property, 
see \cite[Lemma 2.1]{JZ2014}, which implies $x_{\a_k} \to x^*$ as $k \to \infty$. Since, for any sequence 
$\{\a_k\}$ with $\a_k \to 0$, $\{x_{\a_k}\}$ has a subsequence converging strongly to $x^*$ 
as $k\to \infty$, we can conclude $x_\a\to x^*$ as $\a\to 0$. 
\end{proof}

\end{document}